\documentclass[10pt,a4paper,
]{article}
\usepackage[
	a4paper, scale=0.85,
]{geometry}

\usepackage{amsmath}
\usepackage{amsfonts}
\usepackage{amsthm}
\usepackage{amssymb}
\usepackage{latexsym}
\usepackage{fixmath}
\usepackage{mathdots}

\usepackage{array}
\newcolumntype{C}[1]{>{\centering\arraybackslash}p{#1}}
\newcolumntype{R}[1]{>{\raggedleft\arraybackslash}p{#1}}
\newcolumntype{L}[1]{>{\raggedright\arraybackslash}p{#1}}
\usepackage{xcolor}
\usepackage{graphicx}
\usepackage{subfig}
\usepackage{placeins}

\usepackage{siunitx}
\usepackage{tikz}

\usepackage{hyperref}
\hypersetup{colorlinks}  

\usepackage[capitalize]{cleveref}
\Crefname{figure}{Figure}{Figures}
\crefname{equation}{}{}
\Crefname{equation}{Equation}{Equations}

\usepackage[shortlabels]{enumitem}
\setlist[enumerate,2]{label=(\alph*),ref=\theenumi(\alph*)}
\setlist[enumerate,3]{label=\roman*.,ref=\theenumii\roman*}

\usepackage{algorithm,algorithmic}

\usepackage{mathtools}

\DeclarePairedDelimiterXPP\set[1]{}{\{}{\}}{}{

#1}
\DeclarePairedDelimiter\abs{|}{|}
\DeclarePairedDelimiter\N{\|}{\|}

\def\ol{\overline}
\def\ul{\underline}
\def\wtd{\widetilde}
\def\what{\widehat}

\def\ii{\mathrm{i}}

\DeclareMathOperator{\diag}{diag}

\DeclareMathOperator{\rank}{rank}

\DeclareMathOperator{\subspan}{span}
\DeclareMathOperator{\trace}{trace}

\DeclareMathOperator{\F}{F}

\DeclareMathOperator{\T}{T}

\newtheorem{theorem}{Theorem}[section]
\newtheorem{lemma}[theorem]{Lemma}

\theoremstyle{definition}
\newtheorem{example}{Example}[section]

\numberwithin{equation}{section}
\numberwithin{figure}{section}
\numberwithin{table}{section}
\allowdisplaybreaks

\usepackage{multicol}
\usepackage{multirow,bigdelim}
\usepackage{spacingtricks}
\newcommand\clue[2]{\stackrel{\makebox[0pt][c]{\scriptsize #1}}{#2}\;}
\DeclareMathOperator{\nnz}{nnz}
\DeclareMathOperator{\nres}{NRes}
\usepackage{amsbsy}
\newcommand\bs[1]{\boldsymbol{ #1}}
\usepackage{mathrsfs}
\newcommand\op[1]{\mathscr{ #1}}

\title{Flexible fixed-point iteration and its applications for nonsymmetric algebraic Riccati equations}

\author{
	Zhen-Chen Guo\thanks{School of Mathematical Sciences, Beijing University of Posts and Telecommunications, Beijing 102206, China;
		\texttt{e-mail: guozhenchen@bupt.edu.cn}. 
Supported in part by NSFC 12571403, and 12371380.
Part of the work was done when the author worked in School of Mathematical Sciences, Nankai University.
} \and 
Xin Liang\thanks{Corresponding author.
	Beijing Institute of Mathematical Sciences and Applications, Beijing 101408, China;
\texttt{e-mail: liangxin@bimsa.cn}.
Supported in part by NSFC-12371380.
} 
}

\date{\today}

\begin{document}
\maketitle

\begin{abstract}
	In this paper, we reveal the intrinsic Toeplitz structure in the unique stabilizing solution for nonsymmetric algebraic Riccati equations by employing a shift-involved fixed-point iteration,
	and propose an RADI-type method for computing this solution for large-scale equations of this type with sparse and low-rank structure by incorporating flexible shifts into the fixed-point iteration.
	We present a shift-selection strategy, termed generalized Leja shifts, based on rational approximation theory, which is incorporated into the RADI-type method.
	We further discuss important implementation aspects for the method, such as low-rank factorization of residuals, implicit update of large-scale sparse matrices, real arithmetics with complex shifts, and related equations of other type.
	Numerical experiments demonstrate the efficiency of both the proposed method and the introduced shift-selection strategy.
\end{abstract}



\section{Introduction}\label{sec:introduction}
Consider nonsymmetric algebraic Riccati equation (NARE)
\begin{equation}\label{eq:mare}
	\op R(X):=XCX-XD-AX+B=0,
\end{equation}
where $X\in \mathbb{R}^{m\times n}$ is the unknown, and the coefficients are $A\in \mathbb{R}^{m\times m}, D\in \mathbb{R}^{n\times n}$, 
$B\in \mathbb{R}^{m\times n},
C\in \mathbb{R}^{n\times m}$.
The terminology \emph{nonsymmetric} here is historically but improperly used to denote the equation is in its \emph{general} form where no symmetry is imposed on the coefficients.

An important special case of NAREs is a continuous-time algebraic Riccati equation (CARE),
	which requires the additional assumption that $B\preceq 0,C\succeq 0$ (i.e., $B$, $C$ are symmetric negative/positive semi-definite respectively), and $A=D^{\T}$.
CAREs arise in various models related to control theory, such as linear-quadratic optimal regulator design, and $H^2$ and $H^\infty$ controller design for linear systems, see, e.g., \cite{lancasterR1995algebraic, antsaklisM2007linear}.
Another important special case of NAREs is an M-matrix algebraic Riccati equation (MARE, coined by \cite{xueXL2012accurate}), which requires the additional assumption that 
$\mathcal{M} = \begin{bmatrix} D & -C \\ -B & A \end{bmatrix}$ is an M-matrix. 
MAREs play an important role in Markov-modulated fluid queue models \cite{rogers1994fluid,beanOS2010stochastic,latoucheT2009stochastic,foreestMS2003analysis} as well as numerical solutions of transport equations \cite{bellmanW1975introduction,ganapol1992investigation,juang1995existence}.
Apart from the two cases above, NAREs also arise in dynamic game theory, such as Nash games and Stackelberg games \cite{aboukandilB1985analytical,freilingJA1996global,freilingJA1999discrete,kremer2003nonsymmetric,aboukandilFIJ2003matrix}.

In this paper, we mainly focus on large-scale general NAREs with sparse and low-rank structure,
namely $A,D$ are sparse matrices, and
$B = L^B R^B, C= L^C R^C$ with $L^B\in \mathbb{R}^{m\times p}, R^B\in \mathbb{R}^{p\times n},
L^C\in \mathbb{R}^{n\times q}, R^C\in \mathbb{R}^{q\times m}$. 
The structure appears commonly in large-scale NAREs, especially in the scenarios of aforementioned models related to control theory, Markov chains, and transport theory, which enables a low-rank approximation of the solution so that large-scale NAREs can be dealt with \cite{bennerKS2016lowrank}.

Typically, the solutions to NARE \cref{eq:mare} are expressed in terms of the (graph) invariant subspace of the associated Hamiltonian-like matrix 
$
	\mathcal{H}=\begin{bmatrix}
	D & -C \\ B & -A
\end{bmatrix}.
$
For any solution $X$ to \cref{eq:mare}, 
it holds that 
\begin{equation}\label{eq:Hami}
	\mathcal{H}\begin{bmatrix}
			I \\ X
		\end{bmatrix}=\begin{bmatrix}
			D & -C \\ B & -A
		\end{bmatrix}\begin{bmatrix}
			I \\ X
		\end{bmatrix}=\begin{bmatrix}
			I \\ X
		\end{bmatrix}(D-CX),
\end{equation}
which is also referred to as linearization of NARE \cref{eq:mare} and $\mathcal{H}$ is called  its linearizing matrix. 
Indeed, $X$ is a solution to NARE \cref{eq:mare} if and only if
\cref{eq:Hami} holds. 

As with many other types of algebraic Riccati equations, NARE \cref{eq:mare} may admit infinite many solutions, but in many applications, including those mentioned above, the solution of interest is the unique stabilizing solution,
	corresponding to the graph stabilizing invariant subspace,  named after the requirement that the matrix $D-CX$ is stable,  i.e., $\lambda(D-CX)\subset \mathbb{C}_-$,
where $\mathbb{C}_-$ is the open left  half complex plane.  
This exactly coincides with the CARE case.
For the MARE, this unique stabilizing solution coincides with  the so-called minimal nonnegative solution that has componentwise minimal entries among all nonnegative solutions.

Substantial progress has been made on both the theoretical and numerical aspects of CAREs and MAREs, see, e.g., \cite{biniIM2012numerical,huangLL2018structurepreserving} and the references therein.
By contrast, algorithms for computing the unique stabilizing solution of general NAREs are far less developed than those for CAREs and MAREs. 
Existing approaches include generalizations of methods originally designed for CAREs,
such as structure-preserving doubling algorithm (SDA) and its variants, including 
alternating-directional doubling algorithm (ADDA), and Newton-ADI methods \cite{biniIM2012numerical,huangLL2018structurepreserving,bennerKS2016lowrank}. 

Taking inspiration from the theoretical results and numerical algorithms developed in our previous work \cite{guoL2023intrinsic} for discrete-time algebraic Riccati equations (DAREs) and CAREs,
we investigate the intrinsic  structure existing in the unique stabilizing solution for general NAREs.
Specifically, we employ a shift-involved fixed-point iteration to reveal this structure and propose an RADI-type algorithm with a low-rank formulation of the residual of an approximation, achieved by incorporating flexible shifts into the fixed-point iteration.
In detail, we present
1) the Toeplitz-structured analytic form of the stabilizing solution of NAREs;
and
2) a nonsymmetric version of the RADI method, which demonstrates advantages in solving large-scale problems.

The rest of the paper is organized as follows. We first introduce some notation. 
Based on the generalized Cayley transformation, the Toeplitz-structured analytic form of the unique stabilizing solution for NAREs is derived in \cref{sec:the-closed-form-of-the-minimal-nonnegative-solution-to-mares},
where we also provide structured approximations, including an acceleration using many multiple shifts in the generalized Cayley transformation.    
In view of the defect correction technique, the RADI-type method is presented in \cref{sec:incorporation-and-radi}, 
where we derive the low-rank factorization of the residual and discuss some implementation issues, including how to keep both $A$ and $D$ unchanged and how to avoid complex arithmetics for real matrices.   
Several numerical results demonstrating the effectiveness of the proposed RADI-type approach are reported in \cref{sec:experiments-and-discussions}.
Concluding remarks are given in \cref{sec:concluding-remarks}.
Finally, some supporting tedious proofs are provided in the \cref{sec:tedious-proofs}.  

{\bf Notation.}
In the paper, $I_n$, $0_{m\times n}$ and $\bs 1_n$ (or simply $I$, $0$ and $\bs 1$
if the dimension is clear from the context) respectively are the $n\times n$ identity matrix,  $m\times n$ matrix with each entry zero and $n$-dimensional column vector with all entries $1$.   
For a matrix $X$, $X^{\T},\N{X},\N{X}_{\F},\lambda(X), \rho(X)$ are its transpose, spectral norm, Frobenius norm, spectrum and the spectral radius respectively.
By $M \otimes N$  denote the Kronecker product of the matrices $M$ and $N$.
Here is the Sherman-Morrison-Woodbury formula, 
where the inverse sign indicates its feasibility:
\begin{equation}\label{eq:smwf}
	(M + UDV^{\T})^{-1} = M^{-1} - M^{-1} U (D^{-1} + V^{\T} M^{-1} U)^{-1} V^{\T} M^{-1}.
\end{equation}

\section{The analytic form of the stabilizing solution to NAREs}\label{sec:the-closed-form-of-the-minimal-nonnegative-solution-to-mares}
First the existence and uniqueness of the stabilizing solution is guaranteed by
\cite[Theorem~2.5]{biniIM2012numerical}:
if $\mathcal{H}$ has $m$ eigenvalues in $\mathbb{C}\setminus\mathbb{C}_-$ and $n$ eigenvalues in $\mathbb{C}_-$,
and if \cref{eq:Hami} holds with $\lambda(D-CX)\subset \mathbb{C}_-$, then NARE \cref{eq:mare} admits a unique stabilizing solution $X$. 
This statement can be extended to the weakly stabilizing case.
Besides, it is natural to associate NARE
\cref{eq:mare} with its \emph{dual} NARE: $YBY-YA-DY+C=0$, according to 
\begin{equation}\label{eq:H-diagonalized}
	\mathcal{H}\begin{bmatrix}
			I & Y \\ X & I
			\end{bmatrix}=\begin{bmatrix}
				D & -C \\ B & -A
			\end{bmatrix}\begin{bmatrix}
			I & Y \\ X & I
			\end{bmatrix}=\begin{bmatrix}
			I & Y \\ X & I
			\end{bmatrix}\begin{bmatrix}
			D-CX & \\ & -(A-BY)
			\end{bmatrix},
\end{equation}
where $X,Y$ are the solutions to \cref{eq:mare} and its dual, respectively.

The Popov approach, which establishes existence and uniqueness of solutions for CAREs and DAREs, can also be generalized to NARE \cref{eq:mare}; see \cite[Chapter~9]{aboukandilFIJ2003matrix}.  
There, one finds an equivalent condition: NARE \cref{eq:mare} has a unique stabilizing solution if and only if the associated Toeplitz operator has a bounded inverse, provided that $(D,L^C)$ and $(A^{\T},(R^C)^{\T})$ are stabilizable.  
Analogously to the Popov approach for CAREs \cite[Corollary~4.2.2]{ionescuOW1999generalized}, we obtain the following sufficient condition for existence and uniqueness of the stabilizing solution to NARE \cref{eq:mare}:
\begin{equation}\label{eq:assumption:regular-M}
	\fbox{$(D,L^C),(A^{\T},(R^C)^{\T}),(D^{\T},(R^B)^{\T}),(A,L^B)$ are stabilizable.}
\end{equation}
Here, a matrix pair $(M
_{r\times r},N
_{r\times r'})$ is called \emph{stabilizable} if
$\rank\left(\begin{bmatrix} M-\lambda I & N \end{bmatrix}\right) = r$
	for all $\lambda \in \mathbb{C}\setminus \mathbb{C}_-$
\cite[Definition~A.1.4]{aboukandilFIJ2003matrix}.  
Under assumption \cref{eq:assumption:regular-M}, both Toeplitz operators associated with NARE \cref{eq:mare} and its dual possess bounded inverses.  
We will adopt \cref{eq:assumption:regular-M} throughout the remainder of this paper.

\subsection{The generalized Cayley transformation on NAREs}\label{ssec:nare-and-adda}
For CAREs, applying the Cayley transformation to the associated Hamiltonian matrix yields an equivalent DARE.  
Analogously, the generalized Cayley transformation
\begin{equation}\label{eq:generalized-Cayley}
	\op C^{\alpha,\beta}(M)=(\alpha I+M)^{-1}(\beta I-M),
\end{equation}
when applied to $\mathcal{H}$
, transforms NARE \cref{eq:mare} into an equivalent nonsymmetric DARE (NDARE), which is first realized by \cite{wangWL2012alternatingdirectional} for MAREs to derive ADDA and is straightforward to be generalized for NAREs.  
Beyond its theoretical significance, this transformation will prove crucial in the development of effective numerical algorithms.

To illustrate the requirement on $\alpha,\beta$, we present a brief process that transforms NARE \cref{eq:mare} into the equivalent NDARE in \cite{wangWL2012alternatingdirectional}.
Temporarily disregarding invertibility issues, applying \cref{eq:generalized-Cayley}
to $\mathcal{H}$ in \cref{eq:H-diagonalized} yields
	$
		\op C^{\alpha,\beta}(\mathcal{H})\begin{bmatrix}
					I &Y\\ X &I
				\end{bmatrix}
				=\begin{bmatrix}
					I &Y\\ X &I
					\end{bmatrix}\begin{bmatrix}
					-R&\\
					& -S^{-1}\\
				\end{bmatrix}
				,
	$ 
	where $R=-\op C^{\alpha,\beta}(D-CX),S=-\op C^{\beta,\alpha}(A-BY)$.
Considering $X$, we have
\begin{equation}\label{eq:general-Caylay:tmp1}
	\begin{bmatrix}
			D_{-\beta} & -C\\
			B & -A_\beta
		\end{bmatrix}\begin{bmatrix}
			I \\ X
			\end{bmatrix}
			=-(\beta I-\mathcal{H})\begin{bmatrix}
			I \\ X
			\end{bmatrix}=(\alpha I+\mathcal{H})\begin{bmatrix}
			I \\ X
		\end{bmatrix}R
				=\begin{bmatrix}
			D_\alpha & -C\\
			B & -A_{-\alpha}
		\end{bmatrix}\begin{bmatrix}
			I \\ X
		\end{bmatrix}R,
\end{equation}
where $A_{\tau}=\tau I+A$, $D_{\tau}=\tau I+D$ for any $\tau=\pm\alpha,\pm\beta$.

Using the block elementary row transformation 
\[
	\mathcal{T}=\begin{bmatrix}
		I & D_{\alpha}^{-1}C
		\\
		0 & I
	\end{bmatrix}
	\begin{bmatrix}
		I & \\ & -(A_{\beta}- BD_{\alpha}^{-1}C)^{-1}
	\end{bmatrix}
	\begin{bmatrix}
		D_{\alpha}^{-1} & 0 
		\\
		-BD_{\alpha}^{-1} & I
	\end{bmatrix}
	,
\]
\cref{eq:general-Caylay:tmp1} is transformed to
\begin{equation}
	\begin{bmatrix}
		-{E}_0 & 0 \\ -H_0 & I
	\end{bmatrix}
	\begin{bmatrix}
		I \\ X
	\end{bmatrix}
	=
	\begin{bmatrix}
		I & -G_0 \\ 0 & -{F}_0
		\end{bmatrix}\begin{bmatrix}
		I \\ X
	\end{bmatrix}R
	\quad
	\Leftrightarrow
	\quad
	\left\{
	\begin{aligned}
				-E_0&=(I-G_0X)R,\\
				X-H_0&=-F_0XR,
	\end{aligned}
	\right.
\label{eq:trans-for-ssf1-EFGH-inverse}
\end{equation}
where
\begin{align}\label{eq:initial:mare}
		F_0&= F(\alpha,\beta):=-(A_\beta  - B D_\alpha^{-1} C)^{-1}(A_{-\alpha}  - B D_\alpha^{-1} C),
		\\\nonumber
		E_0&= E(\alpha,\beta):=-(D_\alpha - C A_\beta^{-1}B)^{-1}(D_{-\beta} - C A_\beta^{-1}B), 	
		\\\nonumber
		H_0&= H(\alpha,\beta):=(\alpha+\beta)(A_\beta - B D_\alpha^{-1} C)^{-1}BD_\alpha^{-1}, \qquad 
		\\\nonumber
		G_0&= G(\alpha,\beta):=(\alpha+\beta)(D_\alpha - C A_\beta^{-1}B)^{-1} CA_\beta^{-1}.
\end{align}
Ultimately, \cref{eq:trans-for-ssf1-EFGH-inverse}
is reformulated as the following NDARE: 
\begin{equation}\label{eq:ndare}
	X-H_0=F_0X(I-G_0X)^{-1}E_0.
\end{equation}
This is the primal equation to the first standard form in \cite[(3.46a)]{huangLL2018structurepreserving}.
	More on \cref{eq:ndare} is also discussed in \cite{liuXL2022nonlinear} for the case that $E_0,F_0, G_0,H_0$ are all nonnegative matrices.

The matrices whose invertibility must be considered include $\alpha I+\mathcal{H},S, \alpha I+[D-CX], D_\alpha,A_\beta, A_\beta-BD_\alpha^{-1}C,D_\alpha-CA_\beta^{-1}B, I-G_0X$.
Assuming that $A_{\beta}, D_{\alpha}$ are nonsingular, we have:
\begin{enumerate}
	\item $\alpha I+\mathcal{H}$ is nonsingular 
		$\Leftrightarrow$ $A_{-\alpha}-BD_\alpha^{-1}C$ is nonsingular;
	\item $\beta I-\mathcal{H}$ is nonsingular 
		$\Leftrightarrow$ $D_{-\beta}-CA_\beta^{-1}B$ is nonsingular;
	\item $R,S$ are nonsingular $\Leftrightarrow$ $\beta I-\mathcal{H},\alpha I+\mathcal{H}$ are nonsingular;
	\item $\alpha I+\mathcal{H}$ is nonsingular $\Rightarrow$ $\alpha I+[D-CX]$ is nonsingular;
	\item $A_\beta-BD_\alpha^{-1}C$ is nonsingular $\clue{\cref{eq:smwf}}{\Leftrightarrow}$ $D_\alpha-CA_\beta^{-1}B$ is nonsingular;
	\item $I-G_0X$ is nonsingular $\Leftarrow$ $E_0$ is nonsingular $\Leftrightarrow$ both $D_{-\beta}-CA_\beta^{-1}B$ and $D_{\alpha}-CA_{\beta}^{-1}B$ are nonsingular.
\end{enumerate}
As a result, it suffices to choose $\alpha,\beta$ to ensure 
\begin{equation}\label{eq:assumption:alpha-beta:inv}
	\fbox{
$D_\alpha,A_\beta, A_\beta-BD_\alpha^{-1}C, A_{-\alpha}-BD_\alpha^{-1}C,D_{-\beta}-CA_\beta^{-1}B$ are nonsingular.
}
\end{equation}

\subsection{The fixed-point iteration}\label{ssec:the-fixed-point-iteration}

For DAREs and CAREs, \cite{guoL2023intrinsic} established a framework based on a fixed-point iteration to analyze the structure of their stabilizing solutions and derive their analytic forms. 
We apply the same technique here to discuss NARE \cref{eq:mare} and its unique stabilizing solution $X_{\star}$, and further reveal that $X_{\star}$ admits a Toeplitz-structured analytic form.

First, we show that NDARE \cref{eq:ndare} is equivalent to NARE \cref{eq:mare}. 
\begin{lemma}\label{lm:equiv:MARE}
	Given $A,B,C,D,\alpha,\beta$ satisfying \cref{eq:assumption:alpha-beta:inv}, and $E_0,F_0,G_0,H_0$ defined in \cref{eq:initial:mare}, NARE \cref{eq:mare} and NDARE \cref{eq:ndare} share the same solution set.
\end{lemma}

\begin{proof}
As in the discussions in \cref{ssec:nare-and-adda}, assumptions \cref{eq:assumption:alpha-beta:inv} guarantee that any solution to NARE \cref{eq:mare} has to be a solution to NDARE \cref{eq:ndare}.

To the opposite, 
for any solution $X$ to NDARE \cref{eq:ndare},
letting $R=-(I-G_0X)^{-1}E_0$,\!\!
\cref{eq:trans-for-ssf1-EFGH-inverse} holds, 
followed by \cref{eq:general-Caylay:tmp1}.
Noticing $A_{-\alpha}-BD_\alpha^{-1}C=-(\mathcal{H}+\alpha I)/D_\alpha$, $D_{-\beta} - C A_\beta^{-1}B=(\mathcal{H}-\beta I)/(-A_\beta)$,
where $N/M$ denotes the Schur complement of $M$ in $N$,
the invertibility of them ensures the invertibility of $\mathcal{H}+\alpha I,$ and $\mathcal{H}-\beta I$, 
so \cref{eq:general-Caylay:tmp1} gives $\op C^{\alpha,\beta}(\mathcal{H})\begin{bmatrix}
	I\\ X
\end{bmatrix}=\begin{bmatrix}
	I\\ X
\end{bmatrix}R$. 
Then by calculating the first block in both sides of \cref{eq:general-Caylay:tmp1} we have 
$[D-CX]-\beta I=([D-CX]+\alpha I)R$, which is equivalent to 
$(I-D_{\alpha}^{-1}CX)R = I - (\alpha+\beta)D_{\alpha}^{-1}-D_{\alpha}^{-1}CX$.
Thus $(I-D_{\alpha}^{-1}CX)(R - I) = -(\alpha+\beta)D_{\alpha}^{-1}$, showing that $I-D_{\alpha}^{-1}CX$ is nonsingular, or equivalently $D_{\alpha}-CX$ is nonsingular.  
Hence it holds that 
$R=-\op C^{\alpha,\beta}(D-CX)$, 
and then \cref{eq:Hami} holds, 
which concludes that $X$ is a solution to the NARE \cref{eq:mare}.
\end{proof}

Based on NDARE \cref{eq:ndare}, it is natural to consider this fixed-point iteration: 
\begin{equation}\label{eq:fixedpoint:mare}
	X_0=0, \qquad	X_1=H_0, \qquad X_{t+1}=\op D(X_t):= H_0 + F_0 X_{t}
	(I-G_0 X_{t})^{-1}E_0,
\end{equation}
where $E_0,F_0,G_0,H_0$ are defined as in \cref{eq:initial:mare}.
The fixed-point iteration can also start from a nonzero initial $X_0$: 
\begin{equation}\label{eq:fixedpoint:mare:arbitrary:initial}
	X_0=\Gamma^A\Gamma^D, \qquad X_{t+1}=\op D(X_t):= H_0 + F_0 X_{t} (I-G_0 X_{t})^{-1}E_0.
\end{equation}

If the fixed-point iteration \cref{eq:fixedpoint:mare}
does not break down, or equivalently it holds that 
$I-G_0 X_{t}$ remains nonsingular for all $t$,
then it is easy to see that the limit of the 
sequence $\set{X_t}$, if exists, is a solution to both NDARE \cref{eq:ndare} and NARE \cref{eq:mare}. 
A natural question is:
does 
$X_t\to X_\star$, the stabilizing solution to NARE \cref{eq:mare}?

Similarly with \cref{ssec:nare-and-adda}, \cref{eq:fixedpoint:mare} or \cref{eq:fixedpoint:mare:arbitrary:initial}
is equivalent to 
\[
	\begin{bmatrix}
		-{E}_0 & 0 \\ -H_0 & I
	\end{bmatrix}
	\begin{bmatrix}
		I \\ X_{t+1}
	\end{bmatrix}
	=
	\begin{bmatrix}
		I & -G_0 \\ 0 & -{F}_0
		\end{bmatrix}\begin{bmatrix}
		I \\ X_t
	\end{bmatrix}R_t
	\quad\Leftrightarrow \quad
	\begin{bmatrix}
		I \\ X_{t+1}
	\end{bmatrix}
	=
	-\op C^{\alpha,\beta}(\mathcal{H})^{-1}\begin{bmatrix}
		I \\ X_t
	\end{bmatrix}R_t,
\]
where $
R_t=-(I-G_0X_t)^{-1}E_0$.
Such an observation
tells that the fixed-point iteration \cref{eq:fixedpoint:mare}
is essentially the power iteration applied to $-\op C^{\alpha,\beta}(\mathcal{H})^{-1}$ with the initial $\begin{bmatrix}
	I \\ X_0
\end{bmatrix}$.
Hence $R_t$ serves as a normalization factor.

\begin{theorem}\label{thm:fixedpoint:mare}
	Let $X_\star,Y_\star$ be the unique stabilizing solution to NARE \cref{eq:mare} and its dual equation respectively, and 
	$
		R_\star=-\op C^{\alpha,\beta}(D-CX_\star), 
		S_\star=-\op C^{\beta,\alpha}(A-BY_\star).
	$ 
	For almost every pair $\alpha,\beta$ (in the sense of the Lebesgue measure) satisfying \cref{eq:assumption:alpha-beta:inv} and\!\!
\begin{equation}\label{eq:assumption:valid-inv}
	\rho(R_\star)\rho(S_\star)<1,
\end{equation}
and an initial $X_0$ with $I-Y_\star X_0$ nonsingular,
we have
	\begin{enumerate}
		\item In \cref{eq:fixedpoint:mare:arbitrary:initial}, $I-G_0X_t$ is nonsingular for any $t=0,1,\dots$, and $X_t\to X_\star$. 
		\item $I-X_tY_\star$ is nonsingular for any $t=0,1,\dots$, and the error admits the explicit form 
			$
			(I-X_tY_\star)^{-1}(X_t- X_\star) = S_{\star}^t(I-X_0Y_\star)^{-1}(X_0-X_\star) R_{\star}^t
			$.
			This implies linear convergence with asymptotic rate
	$\lim\limits_{t\to\infty}\left( \frac{\N{X_t-X_\star}}{\N{X_0-X_\star}} \right){\vphantom{\frac{1}{1}}}^{1/t}\le \rho(R_{\star})\rho(S_{\star})$.
	\end{enumerate}
\end{theorem}
\begin{proof}
	For Item~1,
	the iteration does not break down, i.e.,
	$I-G_0X_t$ remains nonsingular,
	if and only if for each $t$,
	$P(\alpha,\beta):=\begin{bmatrix}
		I & 0
	\end{bmatrix}[-\op C^{\alpha,\beta}(\mathcal{H})^{-1}]^t\begin{bmatrix}
		I \\ X_0
	\end{bmatrix}
	=R_\star^{-t}(I-Y_\star X_\star)^{-1}\big[I+(I-Y_\star X_\star)R_\star^tY_\star S_\star^t(I-X_\star Y_\star)^{-1}(X_0-X_\star)(I-Y_\star X_0)^{-1}\big](I-Y_\star X_0)
	$ is nonsingular.
By the Gel'fand Formula,
$\lim\limits_{t\to\infty} \N{R_\star^tY_\star S_\star^t}^{1/t}$ $\le \lim\limits_{t\to\infty}\N{R_{\star}^t}^{1/t}\N{S_{\star}^t}^{1/t}\N{Y_\star}^{1/t} \le \rho(R_{\star})\rho(S_{\star})<1$.
Hence for sufficiently large $t$, the related matrices are all nonsingular.
On the other hand, for any finite $t$, $\det P(\alpha,\beta)=0$ forms a low-dimensional manifold whose measure is $0$. 
Therefore, for every $\alpha,\beta$ with \cref{eq:assumption:alpha-beta:inv,eq:assumption:valid-inv} but a zero-measure set, all $I-G_0X_t$ is nonsingular.
	Since \cref{eq:assumption:valid-inv} means that $\abs{\mu_i}>\abs{\mu_j}$ for $\mu_i\in\lambda(R_\star^{-1}),\mu_j\in\lambda(S_\star)$.
	Then
the power iteration converges to the stabilizing invariant subspace, 
i.e., $X_t\to X_\star$.

For Item~2, we have known $R_{\star}=-(I-G_0X_\star)^{-1}E_0$, and similarly $S_{\star}=-(I-H_0Y_\star)^{-1}F_0$.
For any $t$, write $\Delta_t=X_\star -X_t$, and then
by \cref{eq:fixedpoint:mare:arbitrary:initial}, calculation gives
\begin{align*}
	\Delta_{t+1}
	=\op D(X_\star)-\op D(X_t)
	&= -F_0(I-X_tG_0)^{-1}\Delta_tR_{\star}
	,
	\\
	(I-X_{t+1}Y_\star)S_{\star}
	&=-F_0(I-X_tG_0)^{-1}(I-X_tY_\star)
	,
\end{align*}
which implies $I-X_{t+1}Y_\star$ is nonsingular as long as $I-X_tY_\star$ is nonsingular, and 
\begin{align*}
	(I-X_{t+1}Y_\star)^{-1}\Delta_{t+1}=S_{\star}(I-X_tY_\star)^{-1}\Delta_tR_{\star}.
\end{align*}
By induction, $I-X_tY_{\star}$ are nonsingular for all $t$, and $(I-X_tY_\star)^{-1}\Delta_t= S_{\star}^t(I-X_0Y_\star)^{-1}\Delta_0R_{\star}^t$
, which is actually the required equality.
By the Gel'fand Formula,
$\lim\limits_{t\to\infty}\left( \frac{\N{X_t-X_\star}}{\N{X_0-X_\star}} \right)^{1/t}\le \lim\limits_{t\to\infty}\N{R_{\star}^t}^{1/t}\N{S_{\star}^t}^{1/t}=\rho(R_{\star})\rho(S_{\star})$.
\end{proof}

More can be said on $\alpha,\beta$. 
If some $\alpha,\beta$ indeed make the iteration break down at $t$, then we can treat the current approximation $X_t$ as the initial in a new round of the iteration. 
This can be done even if the iteration does not break
down, for it would bring in other advantages, see \cref{ssec:different-shifts}.
On the other hand,
the convergence will not be hurt even if $I-G_0X_t$ is so close to a singular matrix for some
$t$ that its inverse (actually solving linear equations) is not accurate numerically,
because the convergence does not rely on the initial and the error can be treated as a perturbation of the initial.

\subsection{The Toeplitz-structured analytic form of the stabilizing solution}\label{ssec:the-toeplitz-structure-in-the-stabilizing-solution}

Recall that 
$B$ and $C$ have factorizations
$B=L^BR^B,C=L^CR^C$.
Though the factorizations are not essential theoretically, in practice they would be very useful to cut down the computational cost.
Besides,
in many applications, such as examples from transport theory and stochastic fluid flow, $B$ and $C$ are usually of low rank and given in their rank factorizations. 
With such factorizations, any $X_t$ in the sequence generated by 
\cref{eq:fixedpoint:mare} enjoys a Toeplitz structure as revealed in \cref{thm:fixedpoint:mare-toeplitz} below,
which implies the intrinsic Toeplitz structure for the fixed-point $X_{\star}$,
analogous to the cases of DAREs and CAREs \cite{guoL2023intrinsic}. 
This also suggests FFT-based Toeplitz-structured approximations exploiting low displacement rank.

To derive the Toeplitz structure in $X_t$ and $X_{\star}$, first
some new terms are introduced:
	\begin{equation}\label{eq:ADBC-new}
		\begin{aligned}
			\wtd A = -A_{\beta}^{-1} A_{-\alpha}=\op C^{\beta,\alpha}(A), \quad L_B = A_\beta^{-1}L^B, \quad R_C &= (\alpha+\beta) R^CA_\beta^{-1},\qquad Y^A= R^C A_\beta^{-1}L^B\in \mathbb{C}^{q\times p}, \\
			\wtd D = -D_{\alpha}^{-1} D_{-\beta}=\op C^{\alpha,\beta}(D), \quad L_C = D_\alpha^{-1} L^C, \quad R_B &=(\alpha+\beta)R^BD_\alpha^{-1}, \qquad Y^D= R^B D_\alpha^{-1} L^C\in \mathbb{C}^{p\times q}. 
		\end{aligned}
	\end{equation}
	Write $ T^A_1=Y^A, T^D_1=Y^D$ and
	\begin{subequations}\label{eq:noniter:X:mare}
		\begin{align}
			U^D_t&=\begin{bmatrix}
				\wtd D^{t-1} L_C&\cdots &\wtd D^2 L_C &\wtd D L_C& L_C
			\end{bmatrix}
			_{n\times tq},
			&V^A_t &= \begin{bmatrix}
				 R_C\wtd A^{t-1} \\ \vdots \\  R_C\wtd A^2\\  R_C\wtd A \\  R_C
			 \end{bmatrix}
			 _{tq\times m}
			,\\
			U^A_t&=\begin{bmatrix}
				 L_B&\wtd A L_B&\wtd A^2 L_B &\cdots &\wtd A^{t-1} L_B
			 \end{bmatrix}
			 _{m\times tp},
			&V^D_t &= \begin{bmatrix}
				 R_B \\  R_B\wtd D \\  R_B\wtd D^2\\ \vdots \\  R_B\wtd D^{t-1}
			 \end{bmatrix}
			 _{tp\times n}
			,\\
			T^A_t &= 
				\begin{bmatrix}
				Y^A & R_C L_B & R_C\wtd A L_B  & \cdots  & R_C\wtd A^{t-2} L_B \\
                    & Y^A     & R_C L_B        & \ddots  & \vdots              \\
                    &         & Y^A            & \ddots  & R_C\wtd A L_B       \\
                    &         &                & \ddots  & R_C L_B             \\
                    &         &                &         & Y^A                 \\
			\end{bmatrix}_{tq\times tp}
			\hspace*{-23pt}
			,
		&T^D_t &= 
				\begin{bmatrix}
				Y^D                  &         &               &         &     \\
				 R_B L_C             & Y^D     &               &         &     \\
				 R_B\wtd D L_C       & R_B L_C & Y^D           &         &     \\
				\vdots               & \ddots  & \ddots        & \ddots  &     \\
				 R_B\wtd D^{t-2} L_C & \cdots  & R_B\wtd D L_C & R_B L_C & Y^D \\
			 \end{bmatrix}_{tp\times tq}
			\hspace*{-23pt}
				.
		\end{align}
	\end{subequations}
\begin{theorem}\label{thm:fixedpoint:mare-toeplitz}
	Suppose that \cref{eq:assumption:alpha-beta:inv,eq:assumption:valid-inv} hold. 
	The terms of the sequence $\set{X_t}$ generated by \cref{eq:fixedpoint:mare} are
	\begin{equation}\label{eq:noniter:X:X:mare}
		X_t	= U^A_t(I-T^D_tT^A_t)^{-1}V^D_t, \qquad t=1,2,\dots,
	\end{equation}
	in which $I-T^D_tT^A_t$ is nonsingular,
	where $U^A_t,V^D_t,T^A_t,T^D_t$ are defined in \cref{eq:ADBC-new,eq:noniter:X:mare}.

	As a result of Item~1 of \cref{thm:fixedpoint:mare} and \cref{eq:noniter:X:X:mare},
	the unique stabilizing solution $X_{\star}$ has a Toeplitz-structured analytic form:
$ 
	X_{\star}=\op U^A(I-\op T^D\op T^A)^{-1}\op V^D,
$ 
where 
\begin{align*}
	\op U^A&=\begin{bmatrix}
		 L_B & \wtd A L_B & \wtd A^2 L_B & \cdots     
	\end{bmatrix}, \quad 
	&\op V^D &=\begin{bmatrix}
		 R_B  \\  R_B\wtd D   \\  R_B\wtd D^2\\ \vdots 
	\end{bmatrix} 
	,\\
	\op T^D &=\begin{bmatrix}
		Y^D \\
		 R_B L_C   & Y^D\\
		 R_B\wtd D  L_C & R_B L_C   & Y^D\\
		 R_B\wtd D^2  L_C & R_B\wtd D  L_C & R_B L_C   & Y^D\\
		\vdots & \ddots & \ddots & \ddots & \ddots 
	\end{bmatrix},
	\quad
	&\op T^A &=\begin{bmatrix}
                 Y^A    & R_C L_B & R_C\wtd A L_B & R_C\wtd A^2 L_B &\cdots\\
                        & Y^A     & R_C L_B       & R_C\wtd A L_B   &\ddots\\
                        &         & Y^A           & R_C L_B         &\ddots\\
                        &         &               & Y^A             &\ddots\\
						&&&&\ddots \\
	\end{bmatrix}.
\end{align*}
\end{theorem}

\begin{proof}
	First we rewrite $E_0,F_0,G_0,H_0$ by the new introduced terms in \cref{eq:ADBC-new}.
	Note that $A_\beta-BD_\alpha^{-1}C$ is nonsingular,
	which implies $I-T_1^DT_1^A=I-Y^D Y^A =I-R^BD_\alpha^{-1}CA_\beta^{-1}L^B$ is nonsingular and
		$(I-Y^D Y^A)^{-1}
		\clue{\cref{eq:smwf}}{=}I+R^BD_\alpha^{-1}C(A_\beta-BD_\alpha^{-1}C)^{-1}L^B
		$.
	Plugging $B=L^BR^B,C=L^CR^C$ 
	into \cref{eq:initial:mare},
	\begin{equation} \label{eq:H0:low-rank}\tag{2.*a}
		\begin{aligned}[b]
			H_0&=(\alpha+\beta)A_\beta^{-1}(I - L^BR^B D_\alpha^{-1}L^CR^CA_\beta^{-1})^{-1} L^B R^B D_{\alpha}^{-1}
			\\&=
			(\alpha+\beta)A_\beta^{-1}L^B(I - R^B D_\alpha^{-1}L^CR^CA_\beta^{-1}L^B)^{-1}  R^B D_{\alpha}^{-1}
			\\&\clue{\cref{eq:ADBC-new}}{=}
			 L_B(I - Y^D Y^A)^{-1} R_B
			=U_1^A(I-T_1^D T_1^A)^{-1} V_1^D,
		\end{aligned}
	\end{equation}
	which is \cref{eq:noniter:X:X:mare} at $t=1$.
	Similarly, we have
		\begin{align*}
			E_0&=\wtd D+U_1^D(I-Y^A Y^D)^{-1}Y^A V_1^D, 
			\\
			F_0&=\wtd A+U_1^A(I-Y^D Y^A)^{-1}Y^D V_1^A, 
			\\
			G_0&=U_1^D(I-Y^AY^D)^{-1}V_1^A. 
		\end{align*}
	Also,
	\begin{align}\label{eq:E0F0:matrix-multiply}\tag{2.*b}
		E_0=\begin{bmatrix}
			U_1^D(I-Y^A Y^D)^{-1}Y^A&I
		\end{bmatrix}\begin{bmatrix}
			V_1^D \\ \wtd D
		\end{bmatrix},
		\qquad
		F_0=\begin{bmatrix}
			U_1^A & \wtd A
		\end{bmatrix}\begin{bmatrix}
			(I-Y^D Y^A)^{-1}Y^D V_1^A \\ I
		\end{bmatrix}.
	\end{align}

	Then suppose that $I-T^D_tT^A_t$ is nonsingular and \cref{eq:noniter:X:X:mare} holds for $t$, and we are going to show them for $t+1$.
	By the fixed-point iteration \cref{eq:fixedpoint:mare},
	\begin{align*}
		X_{t+1}
		& \, = \, H_0+F_0 U^A_t(I-T^D_tT^A_t)^{-1}V^D_t \left(I-G_0 U^A_t(I-T^D_tT^A_t)^{-1}V^D_t\right)^{-1} E_0
		\\&\, =
	H_0+F_0 U^A_t(I-T^D_tT^A_t)^{-1} \left(I-V^D_tG_0 U^A_t(I-T^D_tT^A_t)^{-1}\right)^{-1}V^D_t E_0
			\\&\, \clue{\cref{eq:H0:low-rank}}{=} \,
		U_1^A(I-Y^D Y^A)^{-1} V_1^D+F_0 U^A_t \left((I-T^D_tT^A_t)-V^D_tG_0 U^A_t\right)^{-1}V^D_t E_0
		\\&\, =\,
			\begin{bmatrix}
				U^A_1& F_0 U^A_t
			\end{bmatrix}
			\begin{bmatrix}
				I-Y^DY^A & \\  & I-T^D_tT^A_t-V^D_tG_0 U^A_t
			\end{bmatrix}^{-1}
			\begin{bmatrix}
				V^D_1\\ V^D_t E_0
			\end{bmatrix}
			\\&\, \clue{\cref{eq:E0F0:matrix-multiply}}{=}\, 
		\begin{multlined}[t]
			\begin{bmatrix}
				U^A_1& \wtd A U^A_t
				\end{bmatrix}\begin{bmatrix}
				I  & (I-Y^DY^A)^{-1}Y^DV^A_1 U^A_t \\  & I \\
			\end{bmatrix}
			\\\cdot\begin{bmatrix}
				I-Y^DY^A & \\  & I-T^D_tT^A_t-V^D_tU_1^D(I-Y^AY^D)^{-1}V_1^A U^A_t
			\end{bmatrix}^{-1}
			\\\cdot\begin{bmatrix}
				I & \\ V^D_tU^D_1(I-Y^AY^D)^{-1}Y^A& I
				\end{bmatrix}\begin{bmatrix}
				V^D_1\\ V^D_t \wtd D
			\end{bmatrix}
		\end{multlined}
		\\&\, =\, \begin{bmatrix}
				U^A_1& \wtd A U^A_t
				\end{bmatrix}\begin{bmatrix}
			I-Y^DY^A & -Y^DV^A_1 U^A_t\\ -V^D_tU^D_1Y^A & I-T^D_tT^A_t-V^D_tU^D_1V^A_1 U^A_t
		\end{bmatrix}^{-1}\begin{bmatrix}
				V^D_1\\ V^D_t \wtd D
			\end{bmatrix}
		\\&\, = \, \begin{bmatrix}
				U^A_1& \wtd A U^A_t
				\end{bmatrix}\left(I-\begin{bmatrix}
				Y^D & \\ V^D_tU^D_1 & T^D_t
				\end{bmatrix}\begin{bmatrix}
				Y^A & V^A_1 U^A_t\\ & T^A_t
		\end{bmatrix}\right)^{-1}\begin{bmatrix}
				V^D_1\\ V^D_t \wtd D
			\end{bmatrix}
		\\&\, = \,U^A_{t+1}(I- T^D_{t+1}T^A_{t+1})^{-1}V^D_{t+1}
		.
	\end{align*}
	In the process we have already known $I-T^D_{t+1}T^A_{t+1}$ is nonsingular.

	Once \cref{eq:noniter:X:X:mare} is obtained, the operator expression for $X_{\star}$ will follow 
	because of $X_t\to X_{\star}$ by \cref{thm:fixedpoint:mare}.   
\end{proof}

It is not difficult to find out that
\cref{thm:fixedpoint:mare-toeplitz} coincides with the so-called decoupled form of $H_k$ in the ADDA introduced in \cite{guoCLL2020decoupled,guoCL2020highly} at $t=2^k$, 
which implies that the ADDA can be regarded as an acceleration of the fixed-point iteration \cref{eq:fixedpoint:mare},
or equivalently the ADDA only computes  the subsequence $X_1,X_2,X_4,\dots,X_{2^k},\dots$.

Furthermore, for a nonzero initial $X_0$, the generated sequence $\set{X_t}$ still preserves a Toeplitz structure, 
and 
\cref{thm:mare:arb} provides the analog of \cref{thm:fixedpoint:mare-toeplitz} for an arbitrary initial term.

\begin{theorem}\label{thm:mare:arb}
	The sequence $\set{X_t}$ generated by \cref{eq:fixedpoint:mare:arbitrary:initial} is given by
	\begin{equation}\label{eq:noniter:X:arb}
		X_t=\begin{bmatrix}
			U_t^A &\wtd A^t\Gamma^A
			\end{bmatrix}\left( 
			I-\begin{bmatrix}
				T_t^D \\ \Gamma^D U_t^D
		\end{bmatrix}\begin{bmatrix}
		T_t^A &  V_t^A  \Gamma^A
	\end{bmatrix}\right)^{-1}\begin{bmatrix}
			V_t^D\\ \Gamma^D \wtd D^t 
		\end{bmatrix}
		,
	\end{equation}
	in which $ 
			I-\begin{bmatrix}
				T_t^D \\ \Gamma^D U_t^D
		\end{bmatrix}\begin{bmatrix}
		T_t^A &  V_t^A  \Gamma^A
	\end{bmatrix}$ is nonsingular,
	where $U_t^D, V_t^A,U_t^A, V_t^D, T_t^D, T_t^A$ are defined by \cref{eq:noniter:X:mare}.
\end{theorem}

Note that \cref{eq:noniter:X:arb} reduces to \cref{eq:noniter:X:X:mare} at $\Gamma^A=0, \Gamma^D=0$.
The proof is quite similar to that for the DAREs \cite{guoL2023intrinsic} and
the details are deferred in \cref{sec:tedious-proofs}.

\subsection{Acceleration by using multiple shifts}\label{ssec:different-shifts}
As we can see, the fixed-point iteration converges slowly.
Since the iteration depends on the pair $(\alpha,\beta)$, referred to
as a \emph{shift},
the convergence would be accelerated by using multiple shifts.
Motivated by this idea, 
the following \emph{Flexible Fixed-Point Iteration} is considered:
\begin{equation}\label{eq:fixedpoint:mare:shifts}
	X_0=\Gamma^A\Gamma^D, \qquad X_{t+1}=\op D_t(X_t):= H_t + F_t X_{t} (I-G_t X_{t})^{-1}E_t,
\end{equation}
where $E_t=E(\alpha_t,\beta_t),F_t=F(\alpha_t,\beta_t),G_t=G(\alpha_t,\beta_t),H_t=H(\alpha_t,\beta_t)$ are defined as in \cref{eq:initial:mare}.
The resulting sequence $\set{X_t}$ retains a similar structure, as shown in \cref{thm:fixedpoint:mare-toeplitz:shifts}.
Before stating the result, we introduce some notation. 
For any scalar or matrix $\lambda$ (with $(\alpha+\lambda)^{-1}$ interpreted as $(\alpha I+\lambda)^{-1}$ in the matrix case), define
	\[
		\op C^{\alpha,\beta}_{i,j}(\lambda)=\prod_{l=1}^j\op C^{\alpha_{i-l},\beta_{i-l}}(\lambda)
		=\prod_{l=1}^j\frac{\beta_{i-l}-\lambda}{\lambda+\alpha_{i-l}},\quad
		\ol{\op C}^{\alpha,\beta}_{i,j}(\lambda)=\frac{\op C^{\alpha,\beta}_{i,j}(\lambda)}{\lambda+\alpha_{i-j-1}},\quad
		\ul{\op C}^{\alpha,\beta}_{i,j}(\lambda)=\frac{\op C^{\alpha,\beta}_{i,j}(\lambda)}{\lambda+\alpha_i},\quad
	\]
with the convention $\prod_{l=1}^0(*)=1$. 
We further set
		\begin{align}\nonumber
			\op V^{\alpha,\beta}_t(\lambda)= \begin{bmatrix}
				 \ol{\op C}^{\alpha,\beta}_{t,0}(\lambda) \\  \ol{\op C}^{\alpha,\beta}_{t,1}(\lambda) \\ \vdots \\  \ol{\op C}^{\alpha,\beta}_{t,t-1}(\lambda)
			 \end{bmatrix}, 
			 \quad
\op T^{\alpha,\beta}_t(\lambda)&=
\begin{bmatrix}
					\frac{\ol{\op C}^{\alpha,\beta}_{t,0}(\lambda)}{\alpha_{t-1}+\beta_{t-1}}                  &         &         &     \\
					\frac{\ol{\op C}^{\alpha,\beta}_{t-1,0}(\lambda)}{\alpha_{t-1}+\lambda}             & \ddots  &         &     \\
					\vdots               & \ddots  & \frac{\ol{\op C}^{\alpha,\beta}_{2,0}(\lambda)}{\alpha_1+\beta_1}     &     \\
					\frac{\ol{\op C}^{\alpha,\beta}_{t-1,t-2}(\lambda)}{\alpha_{t-1}+\lambda} & \cdots  & \frac{\ol{\op C}^{\alpha,\beta}_{1,0}(\lambda)}{\alpha_1+\lambda}  & \frac{\ol{\op C}^{\alpha,\beta}_{1,0}(\lambda)}{\alpha_0+\beta_0}\\
			\end{bmatrix}
			,
\\\nonumber
\ul{\op V}^{\alpha,\beta}_t(\lambda)= \begin{bmatrix}
				\ul{\op C}^{\alpha,\beta}_{t-1,t-1}(\lambda) \\ \vdots \\   \ul{\op C}^{\alpha,\beta}_{1,1}(\lambda) \\  \ul{\op C}^{\alpha,\beta}_{0,0}(\lambda)
			 \end{bmatrix}, 
			 \quad\;\;
			P^{\alpha,\beta}_t&=\begin{bmatrix}
		\alpha_{t-1}                 &         &         &     \\
		\alpha_{t-1}+\beta_{t-1}            & \ddots  &         &     \\
		\vdots               & \ddots  & \alpha_1    &     \\
		\alpha_{t-1}+\beta_{t-1}	 & \cdots  & \alpha_1+\beta_1  & \alpha_0\\
\end{bmatrix},
\\\label{eq:Omega-J}
J_t=\diag(1,-1,\dots,(-1)^{t-1}),
				\;
\Omega_t&=\diag(\alpha_{t-1}+\beta_{t-1},\dots,\alpha_1+\beta_1,\alpha_0+\beta_0).
		\end{align}
With this notation one verifies
\begin{equation}\label{eq:T-inv}
	\op V^{\alpha,\beta}_t(\lambda)=\op T^{\alpha,\beta}_t(\lambda)\Omega_tJ_t\bs 1_t,
		\qquad
J_tP^{\alpha,\beta}_tJ_t+\lambda I_t=[\op T^{\alpha,\beta}_t(\lambda)\Omega_t]^{-1}.
\end{equation}

\begin{theorem}\label{thm:fixedpoint:mare-toeplitz:shifts}
	Suppose that \cref{eq:assumption:alpha-beta:inv,eq:assumption:valid-inv} hold for each shift in $(\alpha_0,\beta_0),\dots,(\alpha_{t-1},\beta_{t-1})$. 
	The sequence $\set{X_t}$ generated by the flexible fixed-point iteration \cref{eq:fixedpoint:mare:shifts} is given by
	\begin{equation}\label{eq:noniter:X:arb:shifts}
		X_t=\begin{bmatrix}
			U_t^A &\op C^{\beta,\alpha}_{t,t}(A)\Gamma^A
			\end{bmatrix}\left( 
			I-\begin{bmatrix}
				T_t^D \\ \Gamma^D U_t^D
		\end{bmatrix}\begin{bmatrix}
		T_t^A &  V_t^A  \Gamma^A
	\end{bmatrix}\right)^{-1}\begin{bmatrix}
			V_t^D\\ \Gamma^D \op C^{\alpha,\beta}_{t,t}(D)
		\end{bmatrix}
		,
	\end{equation}
	in which $ 
			I-\begin{bmatrix}
				T_t^D \\ \Gamma^D U_t^D
		\end{bmatrix}\begin{bmatrix}
		T_t^A &  V_t^A  \Gamma^A
	\end{bmatrix}$ is nonsingular,
	where $U_t^D, V_t^A,U_t^A, V_t^D, T_t^D, T_t^A$ are defined by
	\begin{subequations}\label{eq:noniter:X:mare:shifts}
		\begin{alignat}{2}
			U^D_t&=[\ul{\op V}^{\alpha,\beta}_t(D^{\T})]^{\T}(I_t\otimes L^C),
			&V^A_t &= (\Omega_t\otimes R^C)\ul{\op V}^{\beta,\alpha}_t(A),\\
			U^A_t&=[\op V^{\beta,\alpha}_t(A^{\T})]^{\T}(I_t\otimes L^B),
			&V^D_t &= (\Omega_t\otimes R^B)\op V^{\alpha,\beta}_t(D),\\
			T^A_t &= (\Omega_t\otimes R^C)[\op T^{\beta,\alpha}_t(A^{\T})]^{\T}(I_t\otimes L^B),
			\quad&T^D_t &= (\Omega_t\otimes R^B)\op T^{\alpha,\beta}_t(D)(I_t\otimes L^C).
		\end{alignat}
	\end{subequations}
\end{theorem}
\begin{proof}
	The proof is nearly the same as \cref{thm:mare:arb} and thus omitted.
\end{proof}

The convergence behavior of this iteration depends crucially on the chosen shifts. 
Note that
\[
	\begin{bmatrix}
		I \\ X_t
	\end{bmatrix}
	=\op C^{\alpha,\beta}_{t,t}(\mathcal{H})^{-1}
	\begin{bmatrix}
	I \\ X_0
\end{bmatrix}K
,
\] 
where $K=(-1)^tR_0R_1\cdots R_{t-1}$ is a nonsingular normalization with $R_k=-(I-G_kX_k)^{-1}E_k$ for $k=0, 1, \cdots, t-1$, and $\op C^{\alpha,\beta}_{t,t}(\lambda)
$ is a rational function that is a ratio of two degree-$t$ polynomials.

\emph{Which shifts ensure convergence?} 
An exact characterization is difficult, though by analogy with the CARE case \cite{massoudiOR2016analysis}, one may conjecture a sufficient condition similar to the non-Blaschke condition. 
A detailed analysis is left for future work.

On the other hand, recalling that the fixed-point iteration essentially acts as
a power method applied to\linebreak $-\op C^{\alpha,\beta}(\mathcal{H})^{-1}$, one expects faster convergence when the chosen shifts maximize the spectral separation between stabilizing and anti-stabilizing eigenvalues. 
This motivates the following optimization:
\[
	\inf_{\op C^{\alpha,\beta}_{t,t}}\frac{\max_{\lambda\in \lambda(\mathcal{H})\setminus \mathbb{C}_-}\abs{\op C^{\alpha,\beta}_{t,t}(\lambda)}^{-1}}{\min_{\lambda\in \lambda(\mathcal{H})\cap \mathbb{C}_-}\abs{\op C^{\alpha,\beta}_{t,t}(\lambda)}^{-1}}
=
\left(\sup_{\op C^{\alpha,\beta}_{t,t}}\frac{\min_{\lambda\in \lambda(\mathcal{H})\setminus \mathbb{C}_-}\abs{\op C^{\alpha,\beta}_{t,t}(\lambda)}}{\max_{\lambda\in \lambda(\mathcal{H})\cap \mathbb{C}_-}\abs{\op C^{\alpha,\beta}_{t,t}(\lambda)}}
\right)^{-1}
=:\frac{1}{\kappa_t}
.
\]
The problem of determining the optimal $\op C^{\alpha,\beta}_{t,t}$ and the minimal $\frac{1}{\kappa_t}$ for two given sets (here, $\lambda(\mathcal{H})\setminus \mathbb{C}_-$ and $\lambda(\mathcal{H})\cap \mathbb{C}_-$) is known as \emph{Zolotarev's third problem}, 
named after Zolotarev's work \cite{zolotarev1877application} where the problem is solved for the intervals $[-1,1]$ and $(-\infty, -\gamma]\cup[\gamma,+\infty)$ for $\gamma>1$.
Gonchar \cite{gonchar1968generalized,gonchar1969zolotarev} extended the result to any intervals and a plane condenser (a pair of two disjoint closed sets $E,F$ of which each has a connected complement in $\mathbb{C}$ and whose boundaries satisfy $\partial(E\cup F)=\partial E\cup \partial F$).
The theoretical analysis is highly nontrivial, but roughly speaking, if the smallest enclosing disks of the two sets $\lambda(\mathcal{H})\setminus \mathbb{C}_-,\lambda(\mathcal{H})\cap \mathbb{C}_-$ have radii $a,b$ and center distance $c$, then an exponential convergence is implied:
\[
	\lim_{t\to\infty}\kappa_t^{1/t}=p+\sqrt{p^2-1}, 
	\qquad p=\frac{c^2-a^2-b^2}{2ab},
\]
after some calculations on the ratio of radii of the annulus that are the image of the complement of the two smallest enclosing disks in the complex plane under a proper M\"obius transformation.

In practice, asymptotically optimal rational functions can be approximated by choosing generalized \emph{Leja points} \cite{bagby1969interpolation}, 
generated via a greedy procedure: starting from $-\ol\alpha_0$ and $\beta_0$ with minimal distance, each subsequent $(\alpha_k,\beta_k)$ is selected such that
\[
	\max_{\lambda\in \lambda(\mathcal{H})\setminus \mathbb{C}_-}\abs{\op C^{\alpha,\beta}_{k,k}(\lambda)}=\abs{\op C^{\alpha,\beta}_{k,k}(-\ol\alpha_k)},
	\min_{\lambda\in \lambda(\mathcal{H})\cap \mathbb{C}_-}\abs{\op C^{\alpha,\beta}_{k,k}(\lambda)}=\abs{\op C^{\alpha,\beta}_{k,k}(\beta_k)}.
\]

\section{Defect correction and RADI}\label{sec:incorporation-and-radi}
The defect correction technique is a practical approach for solving matrix equations.
For example,
it leads directly to the efficient RADI method for CAREs \cite{bennerBKS2018radi}. In this section, we extend this idea and develop a variant of RADI for NARE \cref{eq:mare}. 

The basic idea is as follows. Once an approximate solution $\wtd X$ is obtained, we write the exact solution as $X_\star = \wtd X + \Delta$. 
Putting this expression into the original equation yields a new equation in $\Delta$, from which an approximation $\wtd \Delta$ can be computed. The updated solution $\wtd X + \wtd \Delta$ should then provide a better approximation than $\wtd X$.

A critical aspect of the defect correction technique is that the transformed equation for $\Delta$ preserves the structures of the original equation. This is confirmed for NARE \cref{eq:mare} by the following result.

\begin{lemma}\label{lm:incorp:mare}
	Let $\wtd X$ be an approximate solution to NARE \cref{eq:mare}.  
	\begin{enumerate}
		\item $\Delta=X_{\star}-\wtd X$ is a solution to the following matrix equation
			\begin{equation}\label{eq:incorp:mare}
				\Delta C\Delta-\Delta (D-C\wtd X)-(A-\wtd XC)\Delta+\op R(\wtd X)=0.
			\end{equation}
		\item Conversely, if $\Delta$ is a solution to \cref{eq:incorp:mare}, then $\wtd X+\Delta$ is a solution to \cref{eq:mare}. 
		\item There exists a unique stabilizing solution $\Delta_{\star}$ to \cref{eq:incorp:mare}, and $X_{\star} =\wtd X + \Delta_{\star}$. 
	\end{enumerate}
\end{lemma}
\begin{proof}
	Items~1 and~2 are straightforward:
	$0=(\wtd X+\Delta)C(\wtd X+\Delta)-(\wtd X+\Delta)D-A(\wtd X+\Delta)+B
	=\Delta C\Delta-(A-\wtd XC)\Delta-\Delta (D-C\wtd X)+\wtd XC\wtd X-\wtd XD-A\wtd X+B
	$.

	Item~3: note that 
	$ 
		\mathcal{H}_{\wtd X}:=
		\begin{bmatrix}
			D-C\wtd X & -C \\ \op R(\wtd X) & -(A-\wtd X C)
		\end{bmatrix}
		=
		\begin{bmatrix}
			I & 0 \\ -\wtd X & I
		\end{bmatrix}
		\mathcal{H} \begin{bmatrix}
			I & 0 \\ \wtd X & I
		\end{bmatrix}.
	$ 
	Thus both \cref{eq:incorp:mare} and its dual equation admit unique stabilizing solutions.
	Moreover, $\lambda(D-CX_{\star}) \subset \mathbb{C}_{-}$ if and only if $\lambda(D-C\wtd X - C \Delta) \subset \mathbb{C}_{-}$, resulting that $\Delta = \Delta_{\star}$.  
\end{proof}

\subsection{Low-rank factorization of the residual}\label{ssec:low-rank-factorization-of-the-residual}
To apply the defect correction technique within the fixed-point iteration \cref{eq:fixedpoint:mare}, a natural question is how to factorize the residual $\op R(\wtd X)$. The same idea extends to the flexible fixed-point iteration \cref{eq:fixedpoint:mare:shifts}. For $\wtd X = X_t$ generated by \cref{eq:fixedpoint:mare:shifts}, the following theorem provides such a factorization.
\begin{theorem}\label{thm:incorp:mare}
For $X_t$ defined by \cref{eq:fixedpoint:mare:shifts} with $X_0=0$, the residual satisfies $\op R(X_t)=L^B_{t}R^B_{t}$, where 
\begin{subequations}\label{eq:thm:incorp:mare}
	\begin{align}
		L^B_{0}&=L^B,\qquad L^B_{t}=L^B-U^A_t(I-T^D_tT^A_t)^{-1}(\Omega_tJ_t\bs1_t\otimes I_p),\\
		R^B_{0}&=R^B,\qquad R^B_{t}=R^B-(\bs1^{\T}_tJ_t\otimes I_p)(I-T^D_tT^A_t)^{-1}V^D_t,
	\end{align}
\end{subequations}
and $U_t^A, T_t^D, T_t^A, V_t^D$ are as in \cref{eq:noniter:X:mare:shifts} and $\Omega_t,J_t$ are as in \cref{eq:Omega-J}.  
\end{theorem}
\begin{proof}
The proof follows directly once several observations are made.  
Note that \cref{eq:T-inv,eq:noniter:X:mare:shifts} guarantee
\begin{equation}\label{eq:pf:thm:incorp:mare}\tag{3.*a}
	\begin{gathered}
	V^D_tL^C=T^D_t(\Omega_tJ_t\bs1_t\otimes I_p),
	\quad
	R^CU^A_t=(\bs1^{\T}_tJ_t\otimes I_p)T^A_t,
	\\
V^D_tD=(\Omega_tJ_t\bs1_t\otimes I_p)R^B-(\Omega_tJ_tP^{\alpha,\beta}_tJ_t\Omega_t^{-1}\otimes I_p)V^D_t,
	\quad
	AU^A_t=L^B(\bs1^{\T}_tJ_t\otimes I_p)-U^A_t(J_t[P^{\beta,\alpha}_t]^{\T}J_t\otimes I_p),
	\end{gathered}
\end{equation}
and
\begin{equation}\label{eq:TDTA:commute}\tag{3.*b}
	\text{
		$\Omega_tJ_tP^{\alpha,\beta}_tJ_t\Omega_t^{-1}\otimes I_p$ and $T^D_t$ commute, while 
		$J_t[P^{\beta,\alpha}_t]^{\T}J_t\otimes I_p$ and $T^A_t$ commute.
	}	
\end{equation}
Besides, it is easy to examine
\begin{equation}\label{eq:pf:thm:incorp:mare:3}\tag{3.*c}
	P^{\alpha,\beta}_t\Omega_t^{-1}+\Omega_t^{-1}[P^{\beta,\alpha}_t]^{\T}=\bs1_t\bs1^{\T}_t.
\end{equation}
Hence
\begin{align*}
	\op R(X_t)-L^B_{t}R^B_{t}
	&=X_tCX_t-X_tD-AX_t+B-L^B_{t}R^B_{t}
	\\&\clue{\cref{eq:noniter:X:arb:shifts}}{=}
	\begin{multlined}[t]
U^A_t(I-T^D_tT^A_t)^{-1}V^D_tL^CR^CU^A_t(I-T^D_tT^A_t)^{-1}V^D_t
\\ -U^A_t(I-T^D_tT^A_t)^{-1}V^D_tD
-AU^A_t(I-T^D_tT^A_t)^{-1}V^D_t+L^BR^B
\\
-\left[L^B-U^A_t(I-T^D_tT^A_t)^{-1}(\Omega_tJ_t\bs1_t\otimes I_p)\right]\left[R^B-(\bs1^{\T}_tJ_t\otimes I_p)(I-T^D_tT^A_t)^{-1}V^D_t\right]
	\end{multlined}
	\\&\clue{\cref{eq:pf:thm:incorp:mare}
	}{=}
	\begin{multlined}[t]
		U^A_t(I-T^D_tT^A_t)^{-1}T^D_t(\Omega_tJ_t\bs1_t\otimes I_p)(\bs1^{\T}_tJ_t\otimes I_p)T^A_t(I-T^D_tT^A_t)^{-1}V^D_t
\\
-U^A_t(I-T^D_tT^A_t)^{-1}[(\Omega_tJ_t\bs1_t\otimes I_p)R^B-(\Omega_tJ_tP^{\alpha,\beta}_tJ_t\Omega_t^{-1}\otimes I_p)V^D_t]
\\-[L^B(\bs1^{\T}_tJ_t\otimes I_p)-U^A_t(J_t[P^{\beta,\alpha}_t]^{\T}J_t\otimes I_p)](I-T^D_tT^A_t)^{-1}V^D_t
\\
+L^B(\bs1^{\T}_tJ_t\otimes I_p)(I-T^D_tT^A_t)^{-1}V^D_t+U^A_t(I-T^D_tT^A_t)^{-1}(\Omega_tJ_t\bs1_t\otimes I_p)R^B
\\
-U^A_t(I-T^D_tT^A_t)^{-1}(\Omega_tJ_t\bs1_t\otimes I_p)(\bs1^{\T}_tJ_t\otimes I_p)(I-T^D_tT^A_t)^{-1}V^D_t
	\end{multlined}
	\\&=
	\begin{multlined}[t]
		U^A_t(I-T^D_tT^A_t)^{-1}T^D_t(\Omega_tJ_t\bs1_t\bs1^{\T}_tJ_t\otimes I_p)T^A_t(I-T^D_tT^A_t)^{-1}V^D_t
\\
+U^A_t(I-T^D_tT^A_t)^{-1}(\Omega_tJ_tP^{\alpha,\beta}_tJ_t\Omega_t^{-1}\otimes I_p)V^D_t
+U^A_t(J_t[P^{\beta,\alpha}_t]^{\T}J_t\otimes I_p)(I-T^D_tT^A_t)^{-1}V^D_t
\\
-U^A_t(I-T^D_tT^A_t)^{-1}(\Omega_tJ_t\bs1_t\bs1^{\T}_tJ_t\otimes I_p)(I-T^D_tT^A_t)^{-1}V^D_t
	\end{multlined}
	\\&=
	\begin{multlined}[t]
U^A_t(I-T^D_tT^A_t)^{-1}\big[
T^D_t(\Omega_tJ_t\bs1_t\bs1^{\T}_tJ_t\otimes I_p)T^A_t
+(\Omega_tJ_tP^{\alpha,\beta}_tJ_t\Omega_t^{-1}\otimes I_p)(I-T^D_tT^A_t)
\\
+(I-T^D_tT^A_t)(J_t[P^{\beta,\alpha}_t]^{\T}J_t\otimes I_p)
-(\Omega_tJ_t\bs1_t\bs1^{\T}_tJ_t\otimes I_p)
\big](I-T^D_tT^A_t)^{-1}V^D_t
	\end{multlined}
	\\&\clue{\cref{eq:pf:thm:incorp:mare:3}}{=} 
	\begin{multlined}[t]
U^A_t(I-T^D_tT^A_t)^{-1}\Big[
T^D_t(\Omega_tJ_t\bs1_t\bs1^{\T}_tJ_t\otimes I_p)T^A_t
\\
-(\Omega_tJ_tP^{\alpha,\beta}_tJ_t\Omega_t^{-1}\otimes I_p)T^D_tT^A_t
-T^D_tT^A_t(J_t[P^{\beta,\alpha}_t]^{\T}J_t\otimes I_p)
\Big](I-T^D_tT^A_t)^{-1}V^D_t
	\end{multlined}
	\\&\clue{\cref{eq:TDTA:commute}}{=} 
	0
	.
	\qedhere
\end{align*}
\end{proof}

Particularly, when all the shifts $(\alpha_i,\beta_i), i=0,\dots,t-1$ are the same, \cref{thm:incorp:mare} reduces to the factorization of the residual $\op R(X_t)$ for the fixed-point iteration \cref{eq:fixedpoint:mare}.

Motivated by the relation between the RADI method \cite{bennerBKS2018radi} and the underlying Toeplitz structure, as discussed in \cite{guoL2023intrinsic}, we derive the RADI-type method for NAREs. The idea is to apply the incorporation update $X \leftarrow X + \Delta$ and update $\Delta$ and $\op R(\Delta)$ using \cref{eq:noniter:X:arb:shifts} (or \cref{eq:noniter:X:X:mare}) and \cref{eq:thm:incorp:mare}, respectively, with $t=1$.
In detail,
	the substitution 
	\begin{gather*}
		A,D, \alpha,\beta, L^B, R^B
		\leftarrow 
		A_k,D_k, \alpha_k,\beta_k, L_k^B, R_k^B
		,
		\quad t\leftarrow 1,
		\quad
		L_{k+1}^B\leftarrow L_1^B,
		R_{k+1}^B\leftarrow R_1^B,
	\end{gather*}
	will be used in \cref{eq:thm:incorp:mare}.
	First we calculate $U_1^A=L_B,V_1^D=R_B,T_1^A=Y^A,T_1^D=Y^D$ in \cref{eq:thm:incorp:mare} by \cref{eq:noniter:X:mare} and the notation is changed accordingly:
\begin{enumerate}[label={P\arabic*.}]
	\item $A_{k,\beta_k}=A_k+ \beta_k I, D_{k,\alpha_k}=D_k+ \alpha_k I$.
	\item $\what L^X_k=L_{B,k}=A_{k,\beta_k}^{-1}L^B_k, \what R^X_k=\frac{1}{\alpha_k+\beta_k}R_{B,k}=R^B_kD_{k,\alpha_k}^{-1}$.
	\item $ Y^A_k= R^C\what L^X_k, Y^D_k= \what R^X_k L^C$.
	\item 
		$\Upsilon_k=\frac{1}{\alpha_k+\beta_k}(I-Y^D_kY^A_k)$.
\end{enumerate}
Then the incorporation suggests updating $X$ by \cref{eq:noniter:X:X:mare}, $B$ by \cref{eq:thm:incorp:mare}, and $A,D$ by \cref{eq:incorp:mare}:
\begin{enumerate}[resume*]
	\item $\Delta_k = \what L^X_k\Upsilon_k^{-1}\what R^X_k$,
		$X_{k+1}=  X_k+\Delta_k$.
		\item $
			L^B_{k+1}=  L^B_k-\what L^X_k\Upsilon_k^{-1},
			R^B_{k+1}=  R^B_k-\Upsilon_k^{-1}\what R^X_k.
			$
		\item $A_{k+1}=  A_k-\Delta_k C, D_{k+1}=  D_k-C\Delta_k $.
\end{enumerate}
The initialization is given by $
A_0 = A, D_0 = D, X_0 = 0, L^B_0 = L^B, R^B_0 = R^B$.

It is worth noting that the RADI method generates the same sequence $\set{X_k}$ as
\cref{eq:fixedpoint:mare:shifts}, as long as all the shifts are chosen the same. 
Thus, RADI can be viewed as an implementation of \cref{eq:fixedpoint:mare:shifts} that may offers practical numerical advantages. 
In the following, we address several implementation aspects.

\subsection{Implicit update of the square matrices}\label{ssec:implicit-update-of-the-square-matrices}
We next discuss techniques to reduce both storage and computational costs. Recall that $p,q \ll m,n$, and that both $A$ and $D$ are large-scale and sparse. Thanks to the sparse and low-rank structures, it is desirable to avoid updating $A$ and $D$ explicitly during the computation.

Note that
\[
A_{k+1} = A_k - \what L^X_k \Upsilon_k^{-1} Y^D_k R^C,
D_{k+1} = D_k - L^C Y^A_k \Upsilon_k^{-1} \what R^X_k.
\]
Let
\[
L^{\Phi}_{k+1} = \sum_{i=0}^k \what L^X_i \Upsilon_i^{-1} Y^D_i,\quad
R^{\Phi}_{k+1} = \sum_{i=0}^k Y^A_i \Upsilon_i^{-1} \what R^X_i,
\]
with $L^{\Phi}_0 = 0,\ R^{\Phi}_0 = 0$. It follows that
\[
A_k = A - L^{\Phi}_k R^C,\quad D_k = D - L^C R^{\Phi}_k.
\]

Recalling $A_{0,\beta_k} = A + \beta_k I$,
the Sherman–Morrison–Woodbury formula
\cref{eq:smwf} helps to derive
\begin{align*}
	\what L^X_k=A_{k,\beta_k}^{-1}L^B_k
	&= (A_{0,\beta_k}- L^{\Phi}_kR^C)^{-1}L^B_k
	 \clue{\cref{eq:smwf}}{=} 
	\ul{A_{0,\beta_k}^{-1}L^B_k}+\ul{A_{0,\beta_k}^{-1} L^{\Phi}_k}Y^A_k,
	\\
	Y^A_k=R^CA_{k,\beta_k}^{-1}L^B_k
	 &= R^C(A_{0,\beta_k}- L^{\Phi}_kR^C)^{-1}L^B_k
	\\&\clue{\cref{eq:smwf}}{=} 
	R^C\left(A_{0,\beta_k}^{-1}+A_{0,\beta_k}^{-1} L^{\Phi}_k(I-R^CA_{0,\beta_k}^{-1} L^{\Phi}_k)^{-1}R^CA_{0,\beta_k}^{-1}\right)L^B_k
	\\&=
	R^CA_{0,\beta_k}^{-1}L^B_k+R^CA_{0,\beta_k}^{-1} L^{\Phi}_k(I-R^CA_{0,\beta_k}^{-1} L^{\Phi}_k)^{-1}R^CA_{0,\beta_k}^{-1}L^B_k
	\\&= (I-R^C\ul{A_{0,\beta_k}^{-1} L^{\Phi}_k})^{-1}R^C\ul{A_{0,\beta_k}^{-1}L^B_k}, 
\end{align*}
where the main computational cost lies in solving the linear systems $A_{0,\beta_k}^{-1}L_k^B$ and $A_{0,\beta_k}^{-1} L^{\Phi}_k$. 
Analogous formulas hold for $\what R^X_k$ and $Y^D_k$:
\[ 
	\what R^X_k = R^B_k D_{0,\alpha_k}^{-1} + Y^D_k R^{\Phi}_k D_{0,\alpha_k}^{-1},
\quad
Y^D_k = R^B_k D_{0,\alpha_k}^{-1} L^C (I - R^{\Phi}_k D_{0,\alpha_k}^{-1} L^C)^{-1},
\] 
where $D_{0,\alpha_k} = D + \alpha_k I$.

\subsection{Real arithmetics with complex shifts}\label{ssec:real-arithmetics-with-complex-shifts}
Even when all matrices $A, B, C$, and $D$ are real, 
it is often natural and beneficial to employ a complex shift $(\alpha_k, \beta_k)$ in the iteration,
as complex shifts would significantly accelerate convergence. 
However, to ensure efficiency, the use of complex arithmetic should be minimized, 
since one complex addition requires two real additions, and one complex multiplication requires four real multiplications and two real additions (or three real multiplications and five real additions).  
To achieve this reduction, after applying a nonreal shift $(\alpha_k,\beta_k)$, 
we immediately use its conjugate $(\ol\alpha_k,\ol\beta_k)$ in the subsequent step, 
and aim to compute all quantities in the $(k+2)$-th iteration from those in the $k$-th iteration while avoiding complex arithmetic as much as possible. 
In fact, the procedure can be reformulated so that no complex arithmetic is needed at all.

Specifically, suppose $L^B_k, R^B_k$, and $X_k$ are all real, and consider two consecutive steps of the fixed-point iteration
\cref{eq:fixedpoint:mare:shifts} with shifts $(\alpha_k, \beta_k)$ and $(\ol\alpha_k,\ol\beta_k)$ applied to
\[
	\Delta C \Delta - \Delta D_k -A_k
	\Delta + L_k^BR_k^B=0, 
\]
with initial $\Delta_0=0$. 
It follows from \cref{thm:fixedpoint:mare-toeplitz:shifts,thm:incorp:mare} that 
\begin{align*}
	\Delta_2=X_{k+2}-X_k&=U^A_2(I-T^D_2T^A_2)^{-1}V^D_2,\\
	L^B_{k+2}-L^B_k&=-U^A_2(I-T^D_2T^A_2)^{-1}(\Omega_2J_2\bs1_2\otimes I_p),\\
	R^B_{k+2}-R^B_k&=-(\bs1_2^{\T}J_2\otimes I_p)(I-T^D_2T^A_2)^{-1}V^D_2,
\end{align*}
where
\begin{align*}
	U^A_2&=
	\begin{bmatrix}
		(A_k+\ol \beta_k I)^{-1}L_k^B &(A_k+\beta_k I)^{-1} (A_k+\ol \beta_k I)^{-1}(\ol \alpha_k I -A_k)L_k^B
	\end{bmatrix}
	,\\
	T^D_2&=\begin{bmatrix}
		R_k^B(D_k+\ol \alpha_k I)^{-1}L^C &0 \\
		(\alpha_k+\beta_k)R_k^B(D_k+\ol \alpha_k I)^{-1}(D_k+ \alpha_k I)^{-1}L^C & R_k^B(D_k+ \alpha_k I)^{-1}L^C
	\end{bmatrix}
	,\\
	T^A_2&=\begin{bmatrix}
		R^C(A_k+\ol \beta_k I)^{-1}L_k^B &
		(\ol \alpha_k+\ol \beta_k)R^C(A_k+\ol \beta_k I)^{-1}(A_k+ \beta_k I)^{-1}L_k^B \\
		0	& R^C(A_k+ \beta_k I)^{-1}L_k^B
	\end{bmatrix}
	,\\
	V^D_2&=
	\begin{bmatrix}
		(\ol \alpha_k+\ol \beta_k)R_k^B(D_k+\ol \alpha_k I)^{-1}
		\\
		(\alpha_k+\beta_k)R_k^B(D_k+ \alpha_k I)^{-1}(D_k+\ol \alpha_k I)^{-1}(\ol \beta_k I -D_k)
	\end{bmatrix}.
\end{align*}
Temporarily write $(A_k+\beta_k I)^{-1}L_k^B=\what L^B_{\Re}+\ii \what L^B_{\Im}$ with real $\what L^B_{\Re}, \what L^B_{\Im}$.  
Since $Z:=(A_k+\beta_k I)^{-1}(A_k+\ol \beta_k I)^{-1}L_k^B$ is real,
the relation
$(A_k+\ol \beta_k I)Z = \what L^B_{\Re}+\ii \what L^B_{\Im}$ implies
$Z=-\frac{1}{\Im\beta_k}\what L^B_{\Im}$.  
Hence
\begin{align*}
	U^A_2
	&=\begin{bmatrix}
		\what L^B_{\Re}-\ii \what L^B_{\Im}&(A_k+\beta_k I)^{-1}\big((\ol \alpha_k+\ol \beta_k)(A_k+\ol \beta_k I)^{-1}-I\big)L_k^B
	\end{bmatrix} \\&=
	\begin{bmatrix}
		\what L^B_{\Re}-\ii \what L^B_{\Im}&  - \frac{\ol \alpha_k+\ol \beta_k}{\Im\beta_k}\what L^B_{\Im}-\what L^B_{\Re}-\ii \what L^B_{\Im}
	\end{bmatrix} \\&=
	\begin{bmatrix}
		\what L^B_{\Re} & \what L^B_{\Im}
	\end{bmatrix} (\Psi^A\otimes I_p),
\end{align*}
where 
$
	\Psi^A=\begin{bmatrix}
	1 & -1 \\ -\ii  & -\ii - \tfrac{\ol \alpha_k+\ol \beta_k}{\Im\beta_k}
	\end{bmatrix}.
	$
	Letting
	$\Theta^A=\begin{bmatrix}
	-\ii &-\frac{\ol \alpha_k+\ol \beta_k}{\Im\beta_k}\\
	& \ii\\
	\end{bmatrix}$, 
\begin{align*}
	T^A_2
	&=\begin{bmatrix}
		R^C(\what L^B_{\Re}-\ii \what L^B_{\Im}) & -\frac{\ol \alpha_k+\ol \beta_k}{\Im\beta_k}R^C \what L^B_{\Im}
		\\
		0 & R^C(\what L^B_{\Re}+\ii \what L^B_{\Im})
	\end{bmatrix}
	=I_2\otimes R^C\what L^B_{\Re}+\Theta^A\otimes R^C\what L^B_{\Im}
	.
\end{align*}
Similarly, letting 
$R_k^B(D_k+\alpha_k I)^{-1} = \what R^B_{\Re}+\ii \what R^B_{\Im}$, 
and
\begin{gather*}
	\Omega_2=\begin{bmatrix}
	\ol \alpha_k+\ol \beta_k&\\&\alpha_k+\beta_k
\end{bmatrix}
,\Psi^D=\begin{bmatrix}
		1 & -\ii \\ -1  & -\ii - \frac{\ol \alpha_k+\ol \beta_k}{\Im\alpha_k}
\end{bmatrix},\Theta^D=\begin{bmatrix}
	-\ii &\\
	-\frac{\alpha_k+\beta_k}{\Im\alpha_k}& \ii\\
	\end{bmatrix},
	\\
	V^D_2=
	(\Omega_2\Psi^D\otimes I_p)\begin{bmatrix}
			\what R^B_{\Re} \\ \what R^B_{\Im}
		\end{bmatrix}
		,
		\qquad
		T^D_2=
	I_2\otimes \what R^B_{\Re}L^C+\Theta^D\otimes \what R^B_{\Im}L^C
	.
\end{gather*}
	Also $\Psi^A\Theta^A(\Psi^A)^{-1}=-(\Omega_2\Psi^D)^{-1}\Theta^D\Omega_2\Psi^D=\begin{bmatrix}
	0 & 1 \\ -1 & 0
\end{bmatrix}=:\what I_2$. 
After calculations we have $\Delta_2=\begin{bmatrix}
		\what L^B_{\Re} & \what L^B_{\Im}
		\end{bmatrix}\Upsilon_2^{-1}\begin{bmatrix}
	\what R^B_{\Re}\\ \what R^B_{\Im}
	\end{bmatrix}$ where
\begin{align*}
	\Upsilon_2
	&=
(\Omega_2\Psi^D\otimes I_p)^{-1}\left(I-[I_2\otimes \what R^B_{\Re}L^C+\Theta^D\otimes \what R^B_{\Im}L^C][I_2\otimes R^C\what L^B_{\Re}+\Theta^A\otimes R^C\what L^B_{\Im}]\right)(\Psi^A\otimes I_p)^{-1}
	\\&=
	\begin{multlined}[t]
	(\Psi^A\Omega_2\Psi^D\otimes I_p)^{-1}-(\Psi^A\Omega_2\Psi^D)^{-1}\otimes \what R^B_{\Re}L^CR^C\what L^B_{\Re}-(\Omega_2\Psi^D)^{-1}\Theta^D(\Psi^A)^{-1}\otimes \what R^B_{\Im}L^CR^C\what L^B_{\Re}
	\\-(\Omega_2\Psi^D)^{-1}\Theta^A(\Psi^A)^{-1}\otimes \what R^B_{\Re}L^CR^C\what L^B_{\Im}-(\Omega_2\Psi^D)^{-1}\Theta^D\Theta^A(\Psi^A)^{-1}\otimes \what R^B_{\Im}L^C R^C\what L^B_{\Im}
	\end{multlined}
	\intertext{using $\Psi^A\Theta^A(\Psi^A)^{-1}=-(\Omega_2\Psi^D)^{-1}\Theta^D\Omega_2\Psi^D=\begin{bmatrix}
	0 & 1 \\ -1 & 0
\end{bmatrix}=:\what I_2$,}
	&=
	\begin{multlined}[t]
	(\Psi^A\Omega_2\Psi^D\otimes I_p)^{-1}-(\Psi^A\Omega_2\Psi^D)^{-1}\otimes \what R^B_{\Re}L^CR^C\what L^B_{\Re}+\what I_2(\Psi^A\Omega_2\Psi^D)^{-1}\otimes \what R^B_{\Im}L^CR^C\what L^B_{\Re}
	\\-(\Psi^A\Omega_2\Psi^D)^{-1}\what I_2\otimes \what R^B_{\Re}L^CR^C\what L^B_{\Im}+\what I_2(\Psi^A\Omega_2\Psi^D)^{-1}\what I_2\otimes \what R^B_{\Im}L^C R^C\what L^B_{\Im}
	\end{multlined}
	\\&=
(\Psi^A\Omega_2\Psi^D\otimes I_p)^{-1}-[I_2\otimes \what R^B_{\Re}L^C-\what I_2\otimes \what R^B_{\Im}L^C](\Psi^A\Omega_2\Psi^D\otimes I_q)^{-1}[I_2\otimes R^C\what L^B_{\Re}+\what I_2\otimes R^C\what L^B_{\Im}]
	\\&=
(\Psi^A\Omega_2\Psi^D)^{-1}\otimes I_p-\begin{bmatrix}
			\what R^B_{\Re}L^C & -\what R^B_{\Im}L^C \\ \what R^B_{\Im}L^C & \what R^B_{\Re}L^C
		\end{bmatrix}(\Psi^A\Omega_2\Psi^D)^{-1}\otimes I_q\begin{bmatrix}
			R^C\what L^B_{\Re} & R^C\what L^B_{\Im} \\ -R^C\what L^B_{\Im} & R^C\what L^B_{\Re}
		\end{bmatrix}
		.
\end{align*}
Here
\begin{align*}
	(\Psi^A\Omega_2\Psi^D)^{-1}
	&=\begin{bmatrix}
		2\Re(\alpha_k+\beta_k) & -2\Im(\alpha_k+\beta_k)+\frac{\abs{\alpha_k+\beta_k}^2}{\Im\alpha_k}\\
		-2\Im(\alpha_k+\beta_k)+\frac{\abs{\alpha_k+\beta_k}^2}{\Im\beta_k} & -2\Re(\alpha_k+\beta_k)+\frac{\abs{\alpha_k+\beta_k}^2\Re(\alpha_k+\beta_k)}{\Im\alpha_k\Im\beta_k}\\
	\end{bmatrix}^{-1}
	\\&=\frac{\Im\alpha_k\Im\beta_k}{\abs{\alpha_k+\beta_k}^2-4\Im\alpha_k\Im\beta_k}\begin{bmatrix}
		-\frac{2\Re(\alpha_k+\beta_k)}{\abs{\alpha_k+\beta_k}^2}+\frac{\Re(\alpha_k+\beta_k)}{\Im\alpha_k\Im\beta_k}& \frac{2\Im(\alpha_k+\beta_k)}{\abs{\alpha_k+\beta_k}^2}-\frac{1}{\Im\alpha_k}\\
		\frac{2\Im(\alpha_k+\beta_k)}{\abs{\alpha_k+\beta_k}^2}-\frac{1}{\Im\beta_k} & \frac{2\Re(\alpha_k+\beta_k)}{\abs{\alpha_k+\beta_k}^2}\\
	\end{bmatrix}
	.
\end{align*}
On the other hand, with $
		e_1=\begin{bmatrix}
			1\\0
		\end{bmatrix}
		$,
\begin{align*}
	L^B_{k+2}-L^B_k
	&=-\begin{bmatrix}
			\what L^B_{\Re} & \what L^B_{\Im}
		\end{bmatrix} \Upsilon_2^{-1}((\Omega_2\Psi^D)^{-1}\Omega_2J_2\bs1_2\otimes I_p)
	=-\begin{bmatrix}
			\what L^B_{\Re} & \what L^B_{\Im}
		\end{bmatrix} \Upsilon_2^{-1}(e_1\otimes I_p)
	,
	\\
	R^B_{k+2}-R^B_k
	&=-(\bs 1_2^{\T}J_2(\Psi^A)^{-1}\otimes I_p)\Upsilon_2^{-1}\begin{bmatrix}
		\what R^B_{\Re}\\\what R^B_{\Im}
	\end{bmatrix}
	=-(e_1^{\T}\otimes I_p)\Upsilon_2^{-1}\begin{bmatrix}
		\what R^B_{\Re}\\\what R^B_{\Im}
	\end{bmatrix}
	.
\end{align*}

Again
the Sherman-Morrison-Woodbury formula
\cref{eq:smwf} could be used to update $A_k,D_k$ implicitly.
In detail,
$\what L^B_{\Re}, \what L^B_{\Im}$ is the
solution to the real linear system 
\[
	\begin{bmatrix}
		A_k+\Re\beta_kI & - \Im\beta_k I\\
		\Im\beta_k I & A_k+\Re\beta_kI
	\end{bmatrix}\begin{bmatrix}
		\what L^B_{\Re}\\ \what L^B_{\Im}
	\end{bmatrix}
	=\begin{bmatrix}
		L_k^B\\
		0
	\end{bmatrix},
\]
and then
\begin{align*}
	\begin{bmatrix}
		R^C\what L^B_{\Re}\\ R^C\what L^B_{\Im}
	\end{bmatrix}
	&=\begin{bmatrix}
		R^C & \\ & R^C
	\end{bmatrix}\left(\underbrace{\begin{bmatrix}
		A_0+\Re\beta_k I & - \Im\beta_k I\\
		\Im\beta_k I & A_0+\Re\beta_kI
	\end{bmatrix}}_{=:\mathcal{A}}-\begin{bmatrix}
	 L^{\Phi}_kR^C & \\ & L^{\Phi}_kR^C 
\end{bmatrix}\right)^{-1}\begin{bmatrix}
		L_k^B\\ 0
	\end{bmatrix}
	\\&\clue{\cref{eq:smwf}}{=} 
	(I_2\otimes R^C)\left(
		\mathcal{A}^{-1}+\mathcal{A}^{-1}(I_2\otimes  L^{\Phi}_k)
		\left[I-(I_2\otimes R^C)\mathcal{A}^{-1}(I_2\otimes  L^{\Phi}_k)\right]^{-1}
		(I_2\otimes R^C)\mathcal{A}^{-1}\right)\begin{bmatrix}
		L_k^B\\ 0
	\end{bmatrix}
	\\&=
	\left(I-(I_2\otimes R^C)\mathcal{A}^{-1}(I_2\otimes  L^{\Phi}_k)\right)^{-1}
		(I_2\otimes R^C)\mathcal{A}^{-1}\begin{bmatrix}
		L_k^B\\ 0
	\end{bmatrix}
	,
	\\
	\begin{bmatrix}
		\what L^B_{\Re}\\ \what L^B_{\Im}
	\end{bmatrix}
	&=\left(\mathcal{A}-(I_2\otimes  L^{\Phi}_kR^C)\right)^{-1}\begin{bmatrix}
		L_k^B\\ 0
	\end{bmatrix}
	\\&\clue{\cref{eq:smwf}}{=} 
	\left(
		\mathcal{A}^{-1}+\mathcal{A}^{-1}(I_2\otimes  L^{\Phi}_k)
		\left[I-(I_2\otimes R^C)\mathcal{A}^{-1}(I_2\otimes  L^{\Phi}_k)\right]^{-1}
		(I_2\otimes R^C)\mathcal{A}^{-1}\right)\begin{bmatrix}
		L_k^B\\ 0
	\end{bmatrix}
	\\&=
		\mathcal{A}^{-1}\begin{bmatrix}
		L_k^B\\ 0
	\end{bmatrix}+\mathcal{A}^{-1}(I_2\otimes  L^{\Phi}_k)\begin{bmatrix}
		R^C\what L^B_{\Re}\\ R^C\what L^B_{\Im}
	\end{bmatrix}
	,
\end{align*}
in which $\mathcal{A}^{-1}(I_2\otimes  L^{\Phi}_k)$ and $\mathcal{A}^{-1}\begin{bmatrix}
	L_k^B \\ 0
\end{bmatrix}$ are needed to be solved.
In fact, it suffices to solve $\mathcal{A}\begin{bmatrix}
	\what  L^{\Phi}_{\Re} \\ \what  L^{\Phi}_{\Im}
\end{bmatrix}=\begin{bmatrix}
	 L^{\Phi}_k\\ 0
\end{bmatrix}$, because $\mathcal{A}\begin{bmatrix}
	-\what  L^{\Phi}_{\Im} \\ \what  L^{\Phi}_{\Re}
\end{bmatrix}=\begin{bmatrix}
	0\\  L^{\Phi}_k
\end{bmatrix}$.
And $\what R^B_{\Re},\what R^B_{\Im}$ and $\what R^B_{\Re}L^C, \what R^B_{\Im}L^C$ can be solved in the same manner:
\begin{align*}
	\begin{bmatrix}
		\what R^B_{\Re}L^C & \what R^B_{\Im}L^C
	\end{bmatrix}&=
	\begin{bmatrix}
		R_k^B & 0
	\end{bmatrix}\mathcal{D}^{-1}(I_2\otimes L^C)
	\left(I-(I_2\otimes  R^{\Phi}_k)\mathcal{D}^{-1}(I_2\otimes L^C)\right)^{-1},
	\\
	\begin{bmatrix}
		\what R^B_{\Re} & \what R^B_{\Im}
	\end{bmatrix}
	&=\begin{bmatrix}
		R_k^B & 0
		\end{bmatrix}\mathcal{D}^{-1} + \begin{bmatrix}
		\what R^B_{\Re}L^C & \what R^B_{\Im}L^C
	\end{bmatrix}(I_2\otimes  R^{\Phi}_k)\mathcal{D}^{-1},
\end{align*}
where $\mathcal{D}=\begin{bmatrix}
	D_0+\Re \alpha_k I & \Im \alpha_k I \\ -\Im \alpha_k I & D_0+\Re \alpha_k I
\end{bmatrix}$.  

Thus, through careful pairing of conjugate shifts and algebraic manipulation, 
all updates can be expressed entirely in terms of real quantities. 
This ensures that the iteration can fully exploit the convergence benefits of complex shifts while avoiding the computational overhead of complex arithmetic.  

\subsection{Other issues}\label{ssec:other-issues}
Each iteration step implicitly involves verifying the stopping criterion at the end and selecting a shift $(\alpha, \beta)$ at the beginning.
To verify the stopping
criterion,
the approximate accuracy is measured by the normalized residual
\[
	\nu_k=\nres(X_k)=
	\frac{\N{\op R(X_k)}_{\F}}{\N{\op R(0)}_{\F}}
=\frac{\N{L^B_k R^B_k}_{\F}}{\N{B}_{\F}}
=\sqrt{\frac{\trace((L_k^B)^{\T}L_k^BR_k^B(R_k^B)^{\T})}{\trace((L^B)^{\T}L^BR^B(R^B)^{\T})}},
\]
which is easy to compute.

It is fair to say that the choice of shifts $(\alpha,\beta)$ essentially determines the convergence of the proposed method. 
Recall the discussion on global convergence at the end of \cref{ssec:different-shifts}.
In the $k$-th iteration, since the spectrum $\lambda(\mathcal{H}_k)$ is unknown,
we employ the projected spectrum $\lambda(\operatorname{proj}_{\Pi_L,\Pi_R}(\mathcal{H}_k))$ instead, while preserving certain structure. Here
\[
\mathcal{H}_k=\begin{bmatrix}
	D_k & -C \\
	 L^B_kR^B_k  & -A_k
\end{bmatrix}, \qquad
\operatorname{proj}_{\Pi_L,\Pi_R}(\mathcal{H}_k)=\begin{bmatrix}
	\Pi_RD_k\Pi_R^{\T} & -\Pi_RC\Pi_L \\
	\Pi_L^{\T}L^B_kR^B_k   \Pi_R^{\T} & -\Pi_L^{\T}A_k\Pi_L
\end{bmatrix},
\]
and $\Pi_L$ and $\Pi_R^{\T}$ are orthonormal bases of the subspaces 
$\subspan\set{\what  L^X_{k-1},\dots,\what  L^X_{k-s}}$ and 
$\subspan\set{(\what  R^X_{k-1})^{\T},\dots,(\what  R^X_{k-s})^{\T}}$, respectively, for a prescribed integer $s$.
Then the projected NARE is
\[
	\breve{X}\Pi_RC\Pi_L\breve{X}-\breve{X}\Pi_RD_k\Pi_R^{\T}- \Pi_L^{\T}A_k\Pi_L\breve{X}+\Pi_L^{\T}L^B_kR^B_k \Pi_R^{\T} =0.
\]
This leads to a shift-selection strategy that we refer to as \emph{generalized Leja shifts}: 
once no
shifts remain, we compute $\lambda(\operatorname{proj}_{\Pi_L,\Pi_R}(\mathcal{H}_k))$ consisting of $2ps$ eigenvalues and then
generate a sequence of generalized Leja points, which are used as shifts in subsequent iterations.

Another option is the so-called residual Hamiltonian shift \cite{bennerBKS2018radi} (or $H$-shifts \cite{bennerKS2016lowrank}).
The procedure is as follows:
compute all eigenpairs $\left(\lambda,\begin{bmatrix}
	\breve{r}\\ \breve{q}
\end{bmatrix}\right)$ of the associated Hamiltonian matrix 
$\operatorname{proj}_{\Pi_L,\Pi_R}(\mathcal{H}_k)$ with $\N{\breve{r}}^2+\N{\breve{q}}^2=1$,
and then define $-\alpha_0,\dots,-\alpha_k$ as the stable eigenvalues $\lambda$ sorted by decreasing $\N{\breve{q}}$, 
while $\beta_0,\dots,\beta_k$ are the anti-stable eigenvalues sorted by increasing $\N{\breve{q}}$.

In particular, if $A_k, B_k, C, D_k$, and $\Pi_L,\Pi_R$ are real, then $\lambda(\operatorname{proj}_{\Pi_L,\Pi_R}(\mathcal{H}_k))$ (and hence the generalized Leja points) are closed under complex conjugation. 
Thus, the strategy in \cref{ssec:real-arithmetics-with-complex-shifts}, where $(\ol\alpha_k,\ol\beta_k)$ follows the nonreal shift $(\alpha_k,\beta_k)$, coincides with the generalized Leja shifts.

Note that shifts generated at one stage may be reused in multiple iterations.
Since the projected spectrum would not provide accurate approximations, it would be preferable to select only one shift per iteration.
The newly generated shifts tend to be better since the projection is refined progressively.
By the pair $(s, s')$ we denote the generalized Leja shift strategy where $ps$ pairs of shifts are generated but only $s'$ (or $s'+1$ if the $s'$-th is nonreal) are subsequently employed.

\bigskip\noindent
Taking all the above details into account, we propose a practical method for NAREs, summarized in \cref{alg:RADI:mare}.

\subsection{Generalized and special NAREs}\label{ssec:generalized-nare}

\paragraph{Generalized NARE}\label{para:generalized-nare}
All the discussions above remain valid with minor modifications 
for the generalized NARE
\begin{equation}\label{eq:generalized-nare}
	MXCXN-MXD-AXN+B=0,
\end{equation}
which reduces to an (ordinary) NARE if $M,N$ are nonsingular:
\[
	XCX-XD'-A'X+B'=0,
\]
where $D'=DN^{-1}, A'=M^{-1}A, B'=M^{-1}BN^{-1}$.
Thus,
\Cref{alg:RADI:mare} can also be
applied to the generalized NARE, with the given $M,N$ in place of $M=I_m,N=I_n$ in Line~1.

\begin{algorithm}[H]
	\caption{An RADI-type method for the NARE \cref{eq:mare}}\label{alg:RADI:mare}
	\begin{algorithmic}[1]
		\REQUIRE $A,D,B=L^BR^B,C=L^CR^C$ and a strategy to generate shifts.
		\STATE $ L^{\Phi}=0_{m\times q}, R^{\Phi} =0_{q\times n}, M=I_m,N=I_n$.
		\STATE $L^X=[\;], R^X =[\;]$, $\nu_0=\N{L^BR^B}_{\F}$.
		\WHILE{$\N{L^BR^B}_{\F}>\nu_0\varepsilon $ (stopping criterion)}
		\STATE Generate a proper shift $(\alpha,\beta)$.
		\IF{$(\alpha,\beta)\notin \mathbb{R}^2$ and complex arithmetic to be avoided}
		\STATE $\begin{bmatrix}
			\what L^B_{\Re} & \what  L^{\Phi}_{\Re}\\
			\what L^B _{\Im}& \what  L^{\Phi}_{\Im}\\
		\end{bmatrix}=\begin{bmatrix}
		A+\Re\beta M & - \Im\beta M\\
		\Im\beta M & A+\Re\beta M
	\end{bmatrix}^{-1}\begin{bmatrix}
	L^B &  L^{\Phi} \\ 0 & 0
		\end{bmatrix}$.
		\STATE  $\begin{bmatrix}
		Y^A_{\Re} \\ Y^A_{\Im} \\
		\end{bmatrix}=\left(I_{2q}-\begin{bmatrix}
		R^C\what L^{\Phi}_{\Re}& -R^C\what L^{\Phi}_{\Im}\\
		R^C\what L^{\Phi}_{\Im}& R^C\what L^{\Phi}_{\Re}
		\end{bmatrix}\right)^{-1}\begin{bmatrix}
			R^C\what L^B_{\Re} \\
			R^C\what L^B _{\Im}\\
		\end{bmatrix}$. 
		\STATE $\begin{bmatrix}
			\what R^B_{\Re} & \what R^B _{\Im}\\
			\what  R^{\Phi}_{\Re}& \what  R^{\Phi}_{\Im}\\
		\end{bmatrix}=\begin{bmatrix}
	R^B & 0 \\  R^{\Phi}  & 0
		\end{bmatrix}\begin{bmatrix}
		D+\Re\alpha N & \Im\alpha N\\
		- \Im\alpha N & D+\Re\alpha N
	\end{bmatrix}^{-1}$.
	\STATE $\begin{bmatrix}
	Y^D_{\Re} & Y^D_{\Im} \\
		\end{bmatrix}=\begin{bmatrix}
		\what R^B_{\Re}L^C &
			\what R^B _{\Im}L^C
			\end{bmatrix}\left(I_{2q}-\begin{bmatrix}
		\what R^{\Phi}_{\Re}L^C& \what R^{\Phi}_{\Im}L^C\\
		-\what R^{\Phi}_{\Im}L^C&  \what R^{\Phi}_{\Re}L^C
		\end{bmatrix}\right)^{-1}$. 
		\STATE (LU factorization) $L_\Upsilon R_\Upsilon=
		\Psi\otimes I_p-\begin{bmatrix}
			Y^D_{\Re} & -Y^D_{\Im} \\ Y^D_{\Im} & Y^D_{\Re}
		\end{bmatrix}\Psi\otimes I_q\begin{bmatrix}
			Y^A_{\Re} & Y^A_{\Im} \\ -Y^A_{\Im} & Y^A_{\Re}
		\end{bmatrix}
$,\\ where $\Psi= \frac{\Im\alpha\Im\beta}{\abs{\alpha+\beta}^2-4\Im\alpha\Im\beta}\begin{bmatrix}
		-\frac{2\Re(\alpha+\beta)}{\abs{\alpha+\beta}^2}+\frac{\Re(\alpha+\beta)}{\Im\alpha\Im\beta}& \frac{2\Im(\alpha+\beta)}{\abs{\alpha+\beta}^2}-\frac{1}{\Im\alpha}\\
		\frac{2\Im(\alpha+\beta)}{\abs{\alpha+\beta}^2}-\frac{1}{\Im\beta} & \frac{2\Re(\alpha+\beta)}{\abs{\alpha+\beta}^2}\\
	\end{bmatrix}$.
\STATE
$\what L^X= \left(\begin{bmatrix}
	\what L^B_{\Re}& \what L^B_{\Im}
\end{bmatrix}+\begin{bmatrix}
\what L^{\Phi}_{\Re} & \what L^{\Phi}_{\Im}
		\end{bmatrix}\begin{bmatrix}
		Y^A_{\Re}& Y^A_{\Im}\\
		-Y^A_{\Im}& Y^A_{\Re}
		\end{bmatrix}\right)R_\Upsilon^{-1}, L^X\leftarrow\begin{bmatrix}
			 L^X & \what L^X
	\end{bmatrix}$.
		\STATE $\what R^X = L_\Upsilon^{-1}\left(\begin{bmatrix}
				\what R^B_{\Re}\\
				\what R^B_{\Im}
		\end{bmatrix}+\begin{bmatrix}
		Y^D_{\Re}& -Y^D_{\Im}\\
		Y^D_{\Im}& Y^D_{\Re}
		\end{bmatrix}\begin{bmatrix}
		\what R^{\Phi}_{\Re}\\
			\what R^{\Phi}_{\Im}
		\end{bmatrix}\right), R^X \leftarrow\begin{bmatrix}
			 R^X  \\ \what R^X 
	\end{bmatrix}$.
	\STATE 
		$\begin{bmatrix}
			L^B& L^{\Phi}
		\end{bmatrix}
		\leftarrow  \begin{bmatrix}
			L^B& L^{\Phi} 
			\end{bmatrix}+M\what L^X L_\Upsilon^{-1}
				\begin{bmatrix}
					-I_p &Y^D_{\Re}\\0 &Y^D_{\Im}
		\end{bmatrix}$.
		\STATE $
		\begin{bmatrix}
			R^B\\ R^{\Phi} 
		\end{bmatrix}\leftarrow \begin{bmatrix}
			R^B\\ R^{\Phi} 
		\end{bmatrix}+
			\begin{bmatrix}
				-I_p & 0 \\
			Y^A_{\Re} & Y^A_{\Im}
			\end{bmatrix}R_\Upsilon^{-1}\what R^XN
			$.
		\ELSE
		\STATE $\begin{bmatrix}
			\what L^B & \what  L^{\Phi}
			\end{bmatrix}=(A+\beta M)^{-1}\begin{bmatrix}
		L^B &  L^{\Phi}
		\end{bmatrix}$.
		\STATE $Y^A=(I_q-R^C\what  L^{\Phi})^{-1}R^C\what L^B$. 
		\STATE $\begin{bmatrix}
			\what R^B \\ \what  R^{\Phi} 
		\end{bmatrix}=\begin{bmatrix}
		R^B \\  R^{\Phi} 
		\end{bmatrix}(D+\alpha N)^{-1}$.
		\STATE $Y^D=\what R^BL^C(I_q-\what R^{\Phi} L^C)^{-1}$. 
		\STATE (LU factorization) $L_\Upsilon R_\Upsilon=
		\frac{1}{\alpha+\beta}(I_p-Y^DY^A)$.
\STATE
$\what L^X= (\what L^B+\what L^{\Phi} Y^A)R_\Upsilon^{-1}, L^X\leftarrow\begin{bmatrix}
			 L^X & \what L^X
	\end{bmatrix}$.
		\STATE $\what R^X = L_\Upsilon^{-1}(\what R^B+Y^D\what R^{\Phi} ), R^X \leftarrow\begin{bmatrix}
			 R^X  \\ \what R^X 
	\end{bmatrix}$.
	\STATE 
		$\begin{bmatrix}
			L^B& L^{\Phi}
		\end{bmatrix}
		\leftarrow  \begin{bmatrix}
			L^B& L^{\Phi}
		\end{bmatrix}+M\what L^X L_\Upsilon^{-1}\begin{bmatrix}
			-I_p& Y^D
		\end{bmatrix}$.
		\STATE $
		\begin{bmatrix}
			R^B\\ R^{\Phi} 
		\end{bmatrix}\leftarrow \begin{bmatrix}
			R^B\\ R^{\Phi} 
		\end{bmatrix}+\begin{bmatrix}
			-I_p \\ Y^A
		\end{bmatrix}R_\Upsilon^{-1}\what R^XN
			$.
		\ENDIF
		\ENDWHILE
		\ENSURE $ L^X, R^X $ satisfying $X_\star\approx  L^X R^X $.
	\end{algorithmic}
\end{algorithm}
\paragraph{NARE/S
}\label{para:structure-strengthened-nare}

In some applications, $A,D$ take the special form
\[
	A=A'-L^\Phi R^C,\quad D=D'-L^C R^\Phi,
\]
where $A'$ and $D'$ are banded or even diagonal matrices.
We refer to this as an NARE
with a Strengthened structure (NARE/S).
Rewriting
$A=A'_1=A'-L^\Phi_0R^C,D=D'_1=D'-L^CR^\Phi_0$ with $L^\Phi_0=L^\Phi,R^\Phi_0=R^\Phi$,
we can see that the calculation process
on this NARE/S coincides with that on
\[
	XL^CR^CX-XD'-A'X+L^BR^B=0.
\]
Therefore, \Cref{alg:RADI:mare} is applicable to the NARE/S as well, with the given $L^\Phi,R^\Phi$ initialized in Line~1 instead of zeros.

\paragraph{NARE/W
}\label{para:structure-weakly-trengthened-nare}

Furthermore, if $A,D$ are of the special form
\[
	A=A'-L^\Phi R^C-L^A R^A,\quad D=D'-L^C R^\Phi-L^DR^D,
\]
where $A'$ and $D'$ are banded or even diagonal matrices,
which we call an NARE
with a Weakly-strengthened structure (NARE/W),
then we deal with the following NARE
\begin{equation}\label{eq:mare:specialA}
	X\begin{bmatrix}
			L^C & L^D
		\end{bmatrix}\begin{bmatrix}
		I & \\ & 0
		\end{bmatrix}\begin{bmatrix}
			R^C\\ R^A
		\end{bmatrix}X-X\left(D'-\begin{bmatrix}
		L^C & L^D
		\end{bmatrix}\begin{bmatrix}
			R^\Phi\\ R^D
		\end{bmatrix}\right)-\left(A'-\begin{bmatrix}
		L^\Phi & L^A
	\end{bmatrix}\begin{bmatrix}
		R^C \\ R^A
	\end{bmatrix}\right)X+L^BR^B=0,
\end{equation}
which is essentially the same as the NARE/S.
We present a modified fragment of the algorithm to illustrate the iteration for real shifts.
Specifically, Lines~16--24 of \cref{alg:RADI:mare} are replaced by Lines~16${}'$--24${}'$ of \cref{alg:RADI:mare:specialA}.
The iteration for complex shifts is similar and therefore omitted.

\begin{algorithm}[t]
	\caption{Part of the RADI-type method for NARE/W \cref{eq:mare:specialA}}\label{alg:RADI:mare:specialA}
	\begin{algorithmic}[1]
		\makeatletter
		\algsetup{linenodelimiter=${}'$:}
		\setcounter{ALC@line}{15}
		\makeatother
		\STATE $\begin{bmatrix}
			\what L^B & \what  L^{\Phi} & \what L^A
			\end{bmatrix}=(A+\beta M)^{-1}\begin{bmatrix}
			L^B &  L^{\Phi} & L^A
		\end{bmatrix}$.
		\STATE $Y^A=\left(
			I-\begin{bmatrix}
				R^C\what  L^{\Phi} & R^C\what  L^A\\
				R^A\what  L^{\Phi} & R^A\what  L^A\\
		\end{bmatrix}\right)
		^{-1}\begin{bmatrix}
			R^C\what L^B\\R^A\what L^B
		\end{bmatrix}$. 
		\STATE $\begin{bmatrix}
			\what R^B \\ \what  R^{\Phi} \\\what R^D
			\end{bmatrix}=\begin{bmatrix}
			R^B \\  R^{\Phi} \\R^D
		\end{bmatrix}(D+\alpha N)^{-1}$.
		\STATE $Y^D=\what R^BL^C\left(I-\begin{bmatrix}
			\what R^{\Phi} L^C & \what R^{\Phi} L^D \\
			\what R^D L^C & \what R^D L^D \\
		\end{bmatrix}\right)^{-1}$. 
		\STATE (LU factorization) $L_\Upsilon R_\Upsilon=
		\frac{1}{\alpha+\beta}\left(I_p-Y^D\begin{bmatrix}
			I_q & \\ & 0
		\end{bmatrix}Y^A\right)$.
		\STATE
		$\what L^X= \left(\what L^B+\begin{bmatrix}
			\what L^{\Phi} & \what L^A
			\end{bmatrix} Y^A\right)R_\Upsilon^{-1}, L^X\leftarrow\begin{bmatrix}
			L^X & \what L^X
		\end{bmatrix}$.
		\STATE $\what R^X = L_\Upsilon^{-1}\left(\what R^B+Y^D\begin{bmatrix}
			\what R^{\Phi}\\\what R^D
		\end{bmatrix} \right), R^X \leftarrow\begin{bmatrix}
			R^X  \\ \what R^X 
		\end{bmatrix}$.
		\STATE 
		$\begin{bmatrix}
			L^B& L^{\Phi}
		\end{bmatrix}
		\leftarrow  \begin{bmatrix}
			L^B& L^{\Phi}
		\end{bmatrix}+M\what L^X L_\Upsilon^{-1}\begin{bmatrix}
			-I_p& Y^D
		\end{bmatrix}$.
		\STATE $
		\begin{bmatrix}
			R^B\\ R^{\Phi} 
			\end{bmatrix}\leftarrow \begin{bmatrix}
			R^B\\ R^{\Phi} 
			\end{bmatrix}+\begin{bmatrix}
			-I_p \\ Y^A
		\end{bmatrix}R_\Upsilon^{-1}\what R^XN
		$.
	\end{algorithmic}
\end{algorithm}

\section{Experiments and discussions}\label{sec:experiments-and-discussions}
\newcommand\plottime[4][1]{
	\tikz[yscale=1.4,xscale=0.8*#1]{
		\fill[black]           (0,0) rectangle (#2, .1);
		\fill[black!30!white] (#2,0) rectangle (#3, .1);
		\fill[black!60!white] (#3,0) rectangle (#4, .1);
}}

We present numerical results from four examples to illustrate the behavior of \cref{alg:RADI:mare}.
All experiments were conducted in MATLAB 2023b under the Windows 11 Professional 64-bit operating system on a PC with an Intel Core i7-11370H processor at 3.30GHz and 32GB RAM. 

The four test problems are denoted as Lung2$-$, Rail, Rail-Nash, and Transport.
The stopping criteria are: $\nu_k < 10^{-12}$, or reaching $300$ iterations, or $\nu_k \ge 10^{12}$, or encountering NaN values during computation.

In particular, Lung2$-$ is tested with different magnitude and different $(p,q)$'s to show the efficiency and robustness of the proposed method, since in this example $B$ and $C$ are randomly generated.
A triple $(e,p,q)$ will be used to denote the experiment with certain magnitude, $p$, and $q$. 
The other three examples are not tested in this pattern because of their practical origins of $B$ and $C$.

\Cref{tab:results-rail} collects the basic performance data, including iteration counts, the numbers of columns of $L^X$ (denoted by ``dim''), and the running time, of all $9\times 12$ experiments (6 triples in Lung2$-$ and the other 3 examples, 12 shift strategies).

Note that the iteration counts ``ite'' is ``dim'' divided by $p$, so it is not listed.
Additional information is included  in the column ``remark'':
if the stopping criterion on $\nu$ is not met, $\nu_{300}$ or NaN is reported;
otherwise, a three-shade gray bar \plottime[.5]{1}{2}{3} indicates the timings for 1) generating shifts, 2) solving large linear equations in Lines~6, 8, 16, and~18,  and 3) others.

\Cref{fig:results} shows the 
convergence behavior in terms of both time (in seconds) vs.\ accuracy ($\nu$) and dim vs.\ accuracy.
For clarity, some curves are truncated to enhance the visibility of those convergent curves.

\begin{table}[t]
	\centering
\caption{Performance}
	\label{tab:results-rail}
	\footnotesize
	\begin{tabular}{@{}c@{\,}}
		\hline
\\ \hline
\\\hline
		shift \\ \hline
 leja 1   \\
 leja 2   \\
 leja 5   \\
 leja c 1 \\
 leja c 2 \\
 leja c 5 \\
 hami 1   \\
 hami 2   \\
 hami 5   \\
 hami c 1 \\
 hami c 2 \\
 hami c 5 \\
\hline
\\ \hline
\\\hline
		shift \\ \hline
 leja 1   \\
 leja 2   \\
 leja 5   \\
 leja c 1 \\
 leja c 2 \\
 leja c 5 \\
 hami 1   \\
 hami 2   \\
 hami 5   \\
 hami c 1 \\
 hami c 2 \\
 hami c 5 \\
\hline
\\ \hline
\\\hline
		shift \\ \hline
 leja 1   \\
 leja 2   \\
 leja 5   \\
 leja c 1 \\
 leja c 2 \\
 leja c 5 \\
 hami 1   \\
 hami 2   \\
 hami 5   \\
 hami c 1 \\
 hami c 2 \\
 hami c 5 \\
\hline
	\end{tabular}
	\hspace*{-2mm}
	\begin{tabular}{|@{\,}R{5ex}@{\;}R{8.2ex}@{\,}L{13.8ex}@{\,}}
		\hline
		\multicolumn{3}{|l}{Lung2$-$: $(1,10,10)$ } \\ \hline
\multicolumn{2}{|c}{\!\!\!ite \!=\! dim/10}	 & remark \\ \hline
	dim      & time& \plottime[.025]{0}{0}{40}$\,=\,$40s \\\hline
 950 & 30.739 &\quad2.201e+12\\
 2050 & 70.690 &\plottime[.025]{0.608}{61.365}{70.690}\\
 800 & 26.100 &\plottime[.025]{0.249}{22.254}{26.100}\\
 3000 & 107.159 &\quad9.953e-01\\
 3020 & 130.674 &\quad9.953e-01\\
 3010 & 166.929 &\quad1.000e+00\\
 700 & 23.305 &\quad1.313e+12\\
 2690 & 91.590 &\quad2.795e+17\\
 1070 & 35.829 &\plottime[.025]{0.241}{30.681}{35.829}\\
 1640 & 57.264 &\plottime[.025]{3.703}{49.429}{57.264}\\
 3010 & 115.231 &\quad1.843e-06\\
 1870 & 84.430 &\plottime[.025]{23.887}{75.131}{84.430}\\
\hline
		\multicolumn{3}{|l}{Lung2$-$: $(0,10,5)$ } \\ \hline
\multicolumn{2}{|c}{\!\!\!ite \!=\! dim/10}	  & remark \\ \hline
	dim      & time& \plottime[.025]{0}{0}{40}$\,=\,$40s \\\hline
 430 & 12.446 &\quad1.075e+12\\
 330 & 10.093 &\quad4.440e+12\\
 2070 & 63.638 &\plottime[.025]{0.784}{55.371}{63.638}\\
 120 & 3.810 &\quad NaN \\
 210 & 7.536 &\quad NaN \\
 170 & 6.757 &\quad NaN \\
 510 & 14.673 &\quad8.217e+12\\
 840 & 24.505 &\quad8.037e+17\\
 890 & 26.659 &\quad1.140e+13\\
 3000 & 88.856 &\quad1.394e+04\\
 2310 & 79.296 &\plottime[.025]{11.979}{69.551}{79.296}\\
 3010 & 124.039 &\quad1.331e-10\\
\hline
	\multicolumn{3}{|l}{Rail}  \\ \hline
\multicolumn{2}{|c}{\!\!\!ite \!=\! dim/6}  & remark \\ \hline
	dim      & time& \plottime[.1]{0}{0}{10}$\,=\,$10s \\\hline
 330 & 8.605 &\plottime[.1]{0.166}{7.137}{8.605}\\
 336 & 8.531 &\plottime[.1]{0.159}{7.050}{8.531}\\
 438 & 11.586 &\plottime[.1]{0.301}{9.662}{11.586}\\
 246 & 6.749 &\plottime[.1]{0.624}{5.659}{6.749}\\
 414 & 12.171 &\plottime[.1]{1.848}{10.368}{12.171}\\
 558 & 22.271 &\plottime[.1]{8.454}{19.875}{22.271}\\
 402 & 10.556 &\plottime[.1]{0.198}{8.781}{10.556}\\
 318 & 8.450 &\plottime[.1]{0.127}{7.000}{8.450}\\
 576 & 14.615 &\plottime[.1]{0.298}{12.137}{14.615}\\
 1800 & 52.584 &\quad4.931e-09 \\
 1800 & 53.737 &\quad2.230e-07 \\
 1800 & 71.994 &\quad1.240e-06 \\
\hline
	\end{tabular}
	\hspace*{-2mm}
	\begin{tabular}{|@{\,}R{5ex}@{\;}R{8.2ex}@{\,}L{13.8ex}@{\,}}
		\hline
		\multicolumn{3}{|l}{Lung2$-$: $(0,10,10)$ } \\ \hline
\multicolumn{2}{|c}{\!\!\!ite \!=\! dim/10}	  & remark \\ \hline
	dim      & time& \plottime[.025]{0}{0}{40}$\,=\,$40s \\\hline
 2700 & 82.650 &\plottime[.025]{0.787}{71.440}{82.650}\\
 830 & 25.637 &\plottime[.025]{0.268}{21.984}{25.637}\\
 1650 & 51.110 &\plottime[.025]{0.612}{44.207}{51.110}\\
 3020 & 97.126 &\quad1.001e+00 \\
 3010 & 118.176 &\quad9.944e-01 \\
 3010 & 156.807 &\quad9.992e-01 \\
 1520 & 43.131 &\quad1.029e+13 \\
 2880 & 89.606 &\plottime[.025]{0.403}{77.771}{89.606}\\
 3010 & 95.515 &\quad2.942e+11 \\
 2230 & 68.553 &\plottime[.025]{4.108}{58.793}{68.553}\\
 1510 & 51.876 &\plottime[.025]{6.681}{45.368}{51.876}\\
 1600 & 62.963 &\plottime[.025]{16.261}{55.837}{62.963}\\
\hline
		\multicolumn{3}{|l}{Lung2$-$: $(0,5,10)$ } \\ \hline
\multicolumn{2}{|c}{\!\!\!ite \!=\! dim/5}  & remark \\ \hline
	dim      & time& \plottime[.025]{0}{0}{40}$\,=\,$40s \\\hline
 1510 & 86.392 &\quad5.903e-07\\
 345 & 20.500 &\plottime[.025]{0.188}{18.120}{20.500}\\
 1505 & 91.316 &\quad2.517e-02\\
 1500 & 87.046 &\quad9.873e-01\\
 1505 & 100.762 &\quad9.797e-01\\
 1460 & 105.999 &\quad NaN \\
 730 & 43.073 &\quad3.435e+12\\
 390 & 23.718 &\plottime[.025]{0.133}{21.037}{23.718}\\
 1375 & 82.934 &\quad8.566e+12\\
 695 & 42.320 &\plottime[.025]{1.479}{37.634}{42.320}\\
 1500 & 90.971 &\quad2.555e-09\\
 1090 & 78.181 &\plottime[.025]{12.795}{70.943}{78.181}\\
\hline
	\multicolumn{3}{|l}{Rail-Nash}  \\ \hline
\multicolumn{2}{|c}{\!\!\!ite \!=\! dim/12}	  & remark \\ \hline
	dim      & time& \plottime[.025]{0}{0}{40}$\,=\,$40s \\\hline
 876 & 40.322 &\plottime[.025]{0.457}{32.403}{40.322}\\
 876 & 40.662 &\plottime[.025]{0.622}{32.640}{40.662}\\
 1596 & 75.756 &\plottime[.025]{1.517}{61.164}{75.756}\\
 504 & 25.744 &\plottime[.025]{2.876}{21.132}{25.744}\\
 192 & 12.453 &\quad NaN \\
 456 & 37.151 &\quad NaN \\
 888 & 40.622 &\plottime[.025]{0.445}{32.586}{40.622}\\
 840 & 38.971 &\plottime[.025]{0.597}{31.324}{38.971}\\
 1584 & 72.717 &\plottime[.025]{0.912}{58.342}{72.717}\\
 504 & 25.670 &\plottime[.025]{2.866}{21.119}{25.670}\\
 708 & 43.525 &\plottime[.025]{11.175}{37.048}{43.525}\\
 828 & 67.490 &\plottime[.025]{29.796}{59.971}{67.490}\\
\hline
	\end{tabular}
	\hspace*{-2mm}
	\begin{tabular}{|@{\,}R{5ex}@{\;}R{8.2ex}@{\,}L{13.8ex}@{\,}}
		\hline
		\multicolumn{3}{|l}{Lung2$-$: $(-1,10,10)$ } \\ \hline
\multicolumn{2}{|c}{\!\!\!ite \!=\! dim/10}  & remark \\ \hline
	dim      & time& \plottime[.025]{0}{0}{40}$\,=\,$40s \\\hline
 2290 & 78.906 &\plottime[.025]{0.651}{68.431}{78.906}\\
 1870 & 64.907 &\plottime[.025]{0.621}{56.387}{64.907}\\
 3010 & 106.843 &\quad1.055e-06\\
 3000 & 103.195 &\quad4.099e-01\\
 3010 & 128.494 &\quad9.953e-01\\
 3010 & 173.807 &\quad9.986e-01\\
 920 & 29.705 &\quad1.288e+15\\
 650 & 20.956 &\quad4.497e+12\\
 2810 & 96.337 &\quad4.241e+12\\
 3000 & 104.183 &\quad8.033e-09\\
 3010 & 115.599 &\quad1.145e-08\\
 2380 & 109.001 &\quad NaN \\
\hline
		\multicolumn{3}{|l}{Lung2$-$: $(0,5,5)$ } \\ \hline
\multicolumn{2}{|c}{\!\!\!ite \!=\! dim/5}  & remark \\ \hline
	dim      & time& \plottime[.025]{0}{0}{40}$\,=\,$40s \\\hline
 475 & 24.154 &\quad1.298e+12\\
 1055 & 57.685 &\plottime[.025]{0.595}{52.170}{57.685}\\
 705 & 38.342 &\plottime[.025]{0.557}{34.654}{38.342}\\
 1510 & 80.362 &\quad9.922e-01\\
 1500 & 87.448 &\quad9.925e-01\\
 1500 & 107.791 &\quad9.970e-01\\
 845 & 42.879 &\quad3.470e+12\\
 530 & 27.506 &\quad1.549e+12\\
 370 & 19.413 &\quad2.069e+12\\
 1500 & 77.624 &\quad6.185e-07\\
 790 & 44.189 &\plottime[.025]{3.185}{39.738}{44.189}\\
 1270 & 81.198 &\plottime[.025]{13.824}{74.425}{81.198}\\
\hline
	\multicolumn{3}{|l}{Transport} \\ \hline
\multicolumn{2}{|c}{\!\!\!ite \!=\! dim}  & remark \\ \hline
	dim      & time& \plottime[5]{0}{0}{.2}$\,=\,$0.2s \\\hline
 39 & 0.133 &\plottime[5]{0.030}{0.055}{0.133}\\
 40 & 0.134 &\plottime[5]{0.031}{0.056}{0.134}\\
 35 & 0.125 &\plottime[5]{0.030}{0.052}{0.125}\\
 39 & 0.148 &\plottime[5]{0.035}{0.061}{0.148}\\
 49 & 0.207 &\plottime[5]{0.078}{0.109}{0.207}\\
 67 & 0.435 &\plottime[5]{0.256}{0.302}{0.435}\\
 39 & 0.195 &\plottime[5]{0.062}{0.091}{0.195}\\
 40 & 0.139 &\plottime[5]{0.034}{0.059}{0.139}\\
 38 & 0.138 &\plottime[5]{0.031}{0.056}{0.138}\\
 39 & 0.137 &\plottime[5]{0.032}{0.056}{0.137}\\
 41 & 0.177 &\plottime[5]{0.069}{0.094}{0.177}\\
 53 & 0.338 &\plottime[5]{0.191}{0.226}{0.338}\\
\hline
	\end{tabular}

\end{table}
\begin{figure}[tp]
	\centering
	\subfloat[Lung2$-$: $(0,10,10)$]
	{\includegraphics[width=0.8\textwidth]{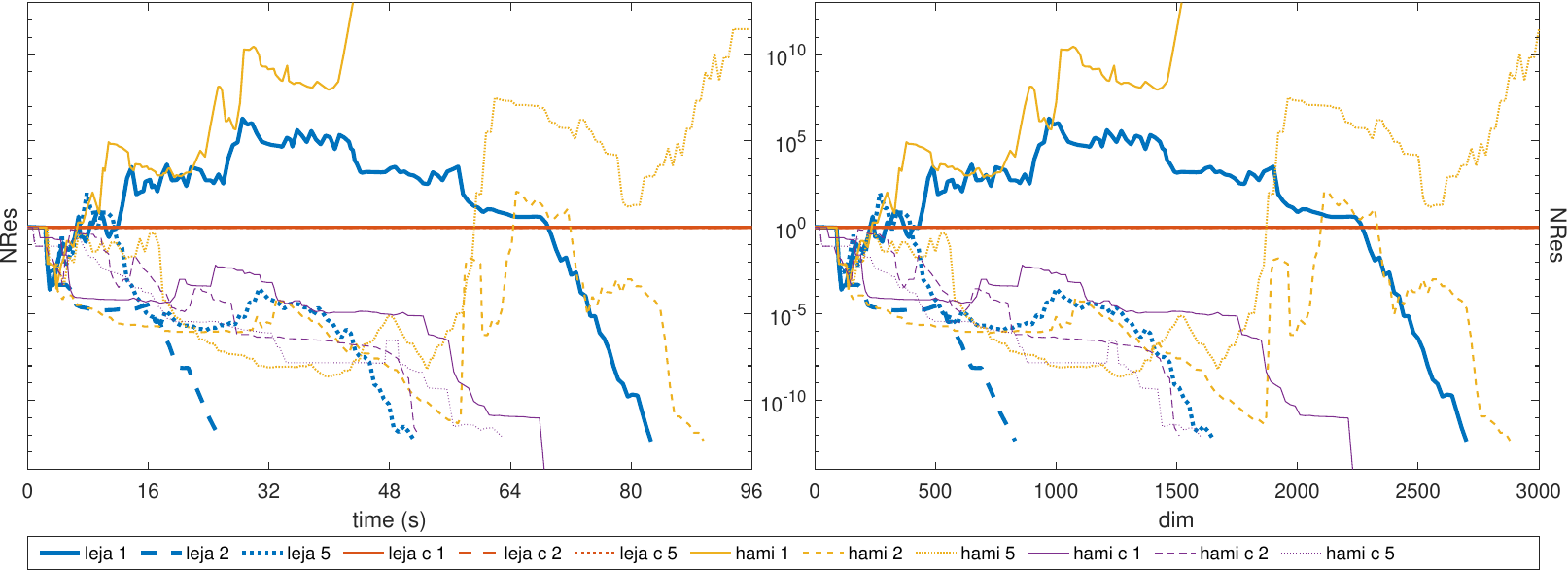}}

	\subfloat[Rail]
	{\includegraphics[width=0.8\textwidth]{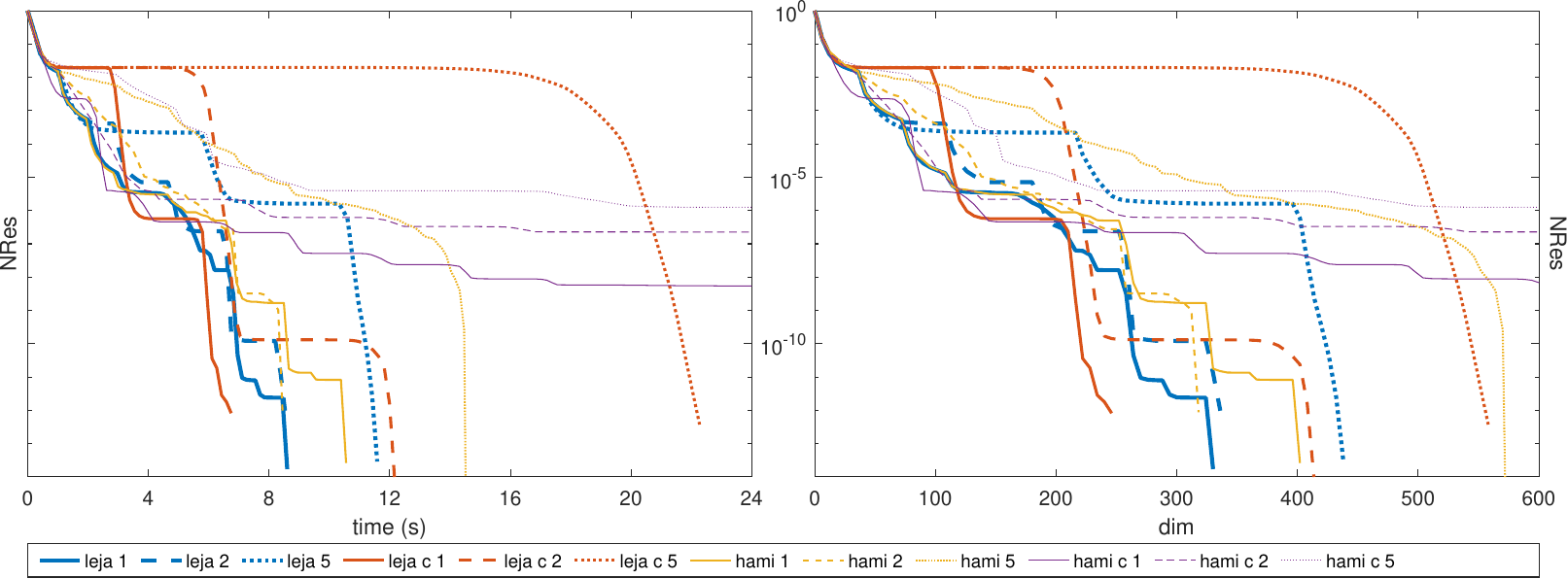}}

	\subfloat[Rail-Nash]
	{\includegraphics[width=0.8\textwidth]{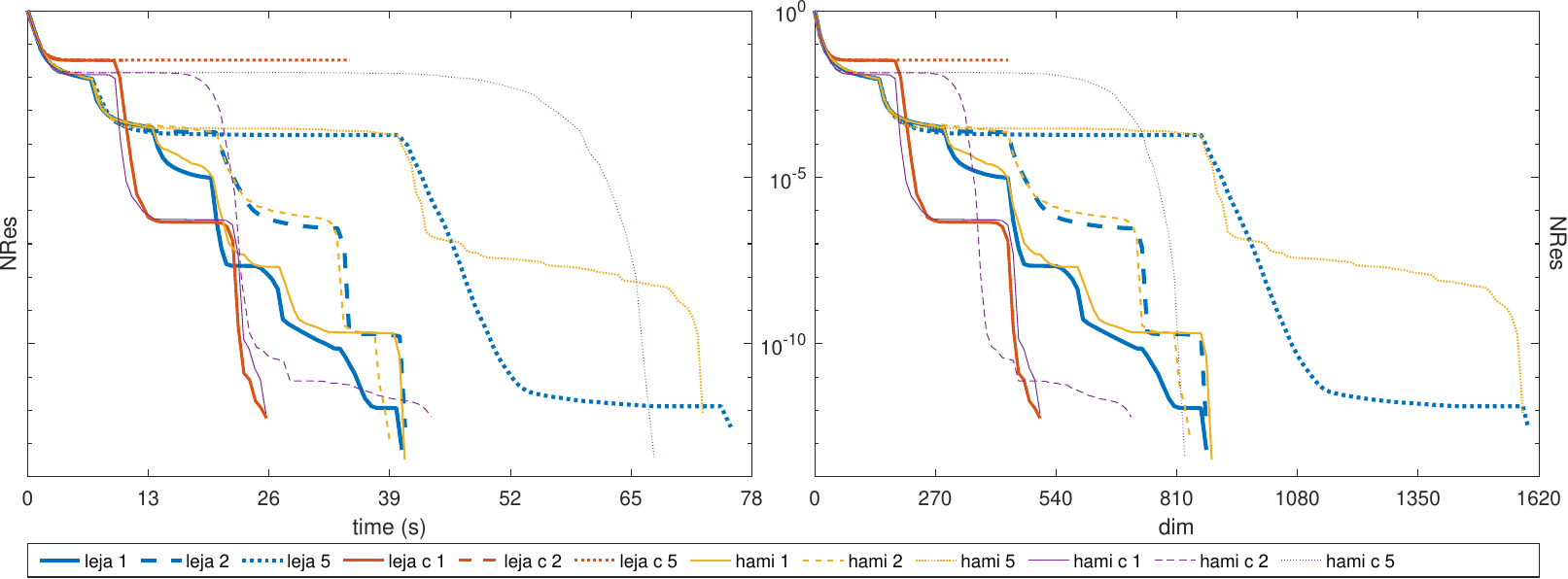}}

	\subfloat[Transport]
	{\includegraphics[width=0.8\textwidth]{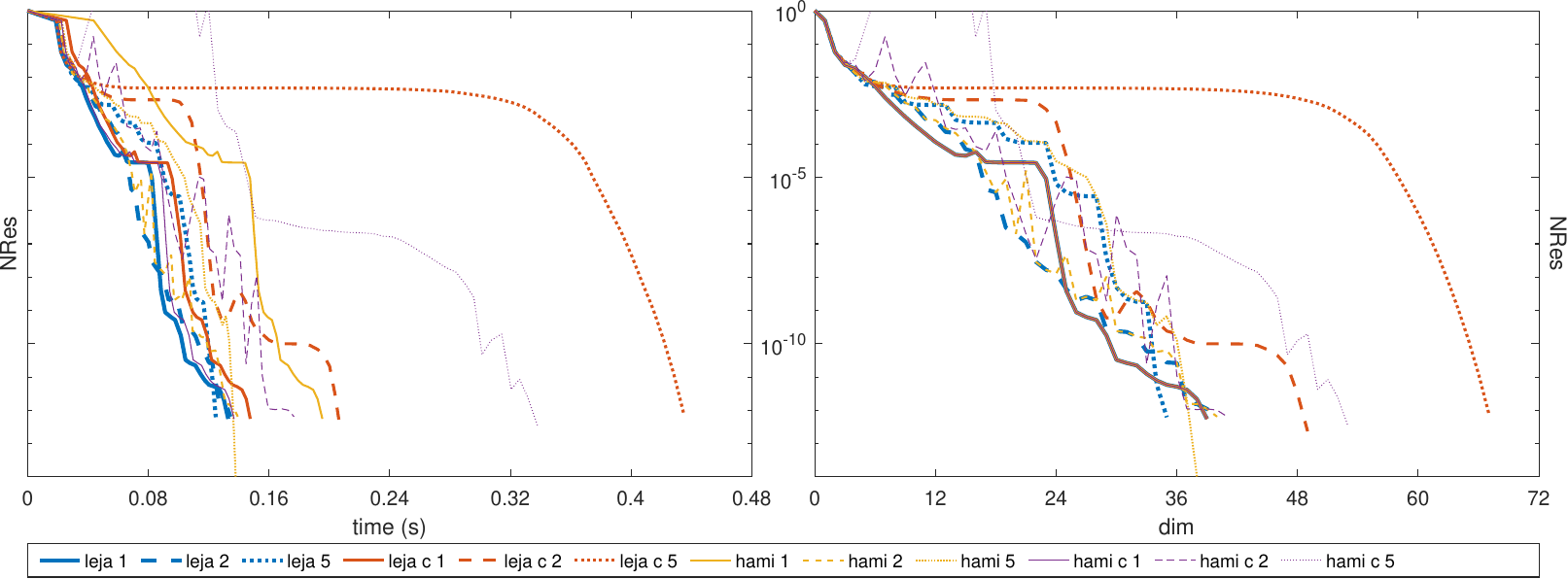}}
\caption{Convergence behavior}
	\label{fig:results}
\end{figure}

For each example, we test four types of shift strategies, discussed in \cref{ssec:other-issues}:
\begin{enumerate}
	\item ``leja $1$/$2$/$5$'': the generalized Leja shifts with the prescribed $s=1,2,5$ respectively, computed once for multiple iterations (one shift per iteration, and recomputed when all shifts are used).
	\item ``leja c $1$/$2$/$5$'': the generalized Leja shifts with the prescribed $s=1,2,5$ respectively, recomputed at every iteration.
	\item ``hami $1$/$2$/$5$'': the residual Hamiltonian shifts with the prescribed $s=1,2,5$ respectively, computed once for multiple iterations.
	\item ``hami c $1$/$2$/$5$'': the residual Hamiltonian shifts with the prescribed $s=1,2,5$ respectively, recomputed at every iteration. 
\end{enumerate}
Note that ``hami c *'' is one of the default choice 
of the package M-M.E.S.S.\ version~2.1 \cite{SaaKB21-mmess-2.1} for RADI applied to CAREs.

The first two examples are (generalized) CAREs.
For the generalized CARE 
\[
\check{C}^{\T}\check{C}+\check{A}^{\T}X\check{E}+\check{E}^{\T}X\check{A}-\check{E}^{\T}X\check{B}\check{B}^{\T}X\check{E}=0 ,
\]
the substitution
$ 
D\leftarrow \check{A}, L^C\leftarrow \check{B}, R^B\leftarrow \check{C}, N\leftarrow \check{E}
$ 
and $A=D^{\T}, R^C=(L^C)^{\T}, L^B=-(R^B)^{\T}$, $M=N^{\T}$ makes it an NARE of the form \cref{eq:generalized-nare}.
Particularly for the case $\check{E}=I$, it is a CARE and the substitution above makes it an NARE of the form \cref{eq:mare}.
Hence we may use \cref{alg:RADI:mare} to solve CAREs.
However, since the involved matrix inverses are $(A+\beta I)^{-1}=(\check{A}^{\T}+\beta I)^{-1}$ and $(D+\alpha I)^{-\!\T}=(\check{A}^{\T}+\alpha I)^{-1}$,
it is reasonable to choose the shifts $\alpha=\beta$ in order to reduce the complexity to inverse the matrix (practically solve the linear equations).
Then \cref{alg:RADI:mare} coincides with the RADI method \cite{bennerBKS2018radi}, despite the shift-selection strategy.
Note that the generalized Leja shifts are tested for the first time.
From the two following examples, the effect of generalized Leja shifts can be observed.

\begin{example}[Lung2$-$]\label{eg:lung2-}
	The example is generated as follows: $-\check{A}$ is the matrix \texttt{lung2} in the SuiteSparse Matrix Collection \cite{davisH2011university} (formerly the University of Florida Sparse Matrix Collection), modelling temperature and water vapor transport in the human lung; $\check{B}_0,\check{C}_0$ are generated by MATLAB function \texttt{rand}.
	Here $m=n=109460,\nnz(\check{A})=492564$ and $\check{B}_0,\check{C}_0$ are dense matrices. The triple $(e,p,q)$ is used to denote the experiment using the first $p$ rows of $10^{e}\check{C}_0$ as $\check{C}$ and the first $q$ columns of $\check{B}_0$ as $\check{B}$.
	The matrix $\check{A}$ is nonsymmetric with 
	$\lambda(\check{A})$ included
	in the left open half-plane, hence stable. 

	The property of the matrix is not very good in the sense that $A$ and $D$ are nonsymmetric, leading to complicated convergence behavior:
	the residual may first grow and later drop to the desired accuracy;
	it may decrease almost to tolerance but suddenly rise;
	or the algorithm may fail to converge under some strategies.
In general, $s=1$ performs worse than $s=2$ or $s=5$, indicating that a very small projection subspace (of dimension $10$) cannot capture the global properties of $\mathcal{H}$.
The best strategy is ``leja 2'' or ``leja 5'', namely $(s,s')=(2,2p)$ or $(5,5p)$.

Compared to results in \cite{bennerBKS2018radi} ($p=q=10$), where $\nu\le 10^{-11}$ was reached in 30 seconds with dim $310$ using ``hami 2'', our tests required about 56 seconds and dim $1820$, possibly due to the different randomly generated $B,C$.

\end{example}

\begin{example}[Rail]\label{eg:rail}
This example is a steel profile cooling model from the Oberwolfach Model Reduction Benchmark Collection, available at MORwiki \cite{morwiki_steel}.
	The data include $\check{A}\prec 0, \check{E}\succ 0, \check{B},\check{C}$ with $m=n=79841,\nnz(\check{A})=553921,p=6,q=7$.
	This is a generalized CARE, so we set $M=N=\check{E}$ instead of $I$ in \cref{alg:RADI:mare}. 
	Since $\check{A}\prec 0$, the system is stable. 

	The property is good in the sense that $A$ and $D$ are symmetric negative definite, and the algorithm converges well.
	The best strategy is ``leja c 1'', namely $(s,s')=(1,1)$.

	Notably, many plateaus appear in the processes ``leja/hami c *''.
	The reason is that the property is so good that every time we calculate the eigenvalues of the projected problem, the smallest several eigenvalues as the shift candidates are nearly the same, which makes the shifts not flexible and thus slowdowns the convergence.

\end{example}

\begin{example}[Rail-Nash]\label{eg:open-loop-nash-riccati-equation}
	The example models a two-player Nash game, leading to an open loop Nash Riccati equation \cite[Section~9.2]{aboukandilFIJ2003matrix}:
	\begin{alignat*}{3}
		D&= A_0,
		&L^C&=\begin{bmatrix}
			B_1 & B_2
		\end{bmatrix},
		&R^B&=\begin{bmatrix}
			C_1\\ C_2
		\end{bmatrix},
		\\
		A&=\begin{bmatrix}
			A_0 & \\ & A_0
		\end{bmatrix}^{\T},
		&\quad R^C&=\begin{bmatrix}
			B_1 &\\ & B_2
		\end{bmatrix}^{\T},
		&\quad L^B&=-\begin{bmatrix}
			C_1&\\ &C_2
		\end{bmatrix}^{\T},
	\end{alignat*}
	where $C_1,C_2$ are associated with two cost functions, and $B_1,B_2$ are associated with two players.
	To use familiar data, the underlying system is the same as Rail:
	$A_0$ matches $\check{A}$ in Rail, while $B_1,C_1$ match $\check{B},\check{C}$ in Rail;
	$B_2$ shares the same scale and sparse pattern with $B$, generated by MATLAB function \texttt{sprand(B)*max(abs(B(:)))}, so does $C_2$.
	Here $m/2=n=79841,\nnz(A)/2=\nnz(D)=553921,p=12,q=14$.
	Also $A\succ0,D\succ0$, ensuring stability.

	Although $A\ne D$, it is still reasonable to set $\alpha=\beta$, since this reduces computational cost by reusing factorizations.
	However, to mimic the $A,D$ distinction, we solve the corresponding systems separately.

	The property is good in the sense that $A$ and $D$ are symmetric negative definite, and the algorithm converges well.
	The best strategies are ``leja/hami c 1'', namely $(s,s')=(1,1)$.

\end{example}

\begin{example}[Transport]\label{eg:mare-transport-theory}
This example concerns the one-group neutron transport equation \cite{juangC1993iterative}:
	\[
		A=\frac{\Omega^{-1}}{\beta(1+\alpha)}-\bs1_nh^{\T},
		D=\frac{\Omega^{-1}}{\beta(1-\alpha)}-h\bs1_n^{\T},
		L^B=(R^B)^{\T}=\bs1_n,
		L^C=(R^C)^{\T}=h,
	\]
where $\Omega=\diag(\omega_1,\dots,\omega_n)$ with $1>\omega_1>\cdots>\omega_n>0$ generated by MATLAB function \texttt{rand}, 
	$h=\frac{1}{2}\Omega^{-1}c$,
	and $c$ is a weight vector with $c_i>0,\sum_{i=1}^nc_i=1$.
	The resulting NARE is an MARE, and also an example on the NARE/S, and $A',D'$ are diagonal.
	We use $n=20000$ with $p=q=1$.

	The property is quite good in the sense that $A'$ and $D'$ are diagonal M-matrices, and the algorithm converges rapidly.
	Strategies ``leja/hami (c) 1'', namely $(s,s')=(1,1)$, all coincide, producing identical shifts and dimensions.
	Since the total execute time is too small, comparisons are not meaningful.

	In \cite{bennerKS2016lowrank}, the low-rank Newton-ADI method used at least 2.8 seconds to reach $\nu\le 10^{-9}$, significantly slower than our method.
They noticed the low-rank structure and used the Sherman-Morrison-Woodbury Formula, but did not establish our technique used for NARE/S that completely eliminates certain costly computations.
	As we can see, unlike the other three examples, the time for solving linear equations is relatively much less, benefiting from the fact that the matrix to be inverted in the NARE/S is diagonal in this example.
	We believe that this is the reason why their method is much slower except the difference of the methods. 

\end{example}

From the above experiments, we recommend the ``leja c 1'' strategy (generalized Leja shifts with $(s,s')=(1,1)$) when the property of the problem is good, and ``leja 2'' or ``leja 5'' (generalized Leja shifts with $(s,s')=(2,2p)$ or $(5,5p)$) otherwise.

\section{Concluding remarks}\label{sec:concluding-remarks}
We have introduced a shift-involved fixed-point iteration and a flexible-shift variant for large-scale (nonsymmetric) algebraic Riccati equations.
This leads to an analytic series form of the stabilizing solution,
and an RADI-type method for computing the unique stabilizing solution, supported by a low-rank factorization of a structured residual.
Furthermore, we proposed a novel shift-selection strategy based on generalized Leja points from rational approximation theory.

In particular, it will be not surprising to discover that the nonnegativeness is preserved in the implementation of the algorithm for MAREs under proper conditions, e.g., the associated M-matrix $\mathcal{M}$ is nonsingular, each shift $\alpha$ is chosen to be real and no smaller than the largest diagonal entry of $A$, and $\beta$ for $D$ (see \cite{wangWL2012alternatingdirectional}).
Since it is too specific to be shown here easily, the details are expected to be presented elsewhere.

An interesting future application lies in the Bethe-Salpeter eigenvalue problem arising in electronic structure calculations, which can be reformulated as an NARE.
According to \cite{bennerDKK2017fast}, this falls into the special NARE/W class.
However, since the target is to compute eigenpairs rather than the stabilizing subspace, it remains unclear whether our method would outperform Lanczos methods modified accordingly that are commonly used.
We plan to investigate this direction in future work.

\section*{Acknowledgements}
The authors thank the anonymous referees for their valuable comments and suggestions.

During the preparation of this work the authors used ChatGPT to improve readability.
After that 
the authors reviewed and edited the content as needed and take full responsibility for the content of the publication.

The authors declared that they have no conflicts of interest to this work.

\appendix
\section{Tedious proofs}\label{sec:tedious-proofs}
\begin{proof}[{Proof of \cref{thm:mare:arb}}]
		First examine $X_1$. 
	\begin{align*}
		X_1&=\, H_0 + F_0 \Gamma^A\Gamma^D(I-G_0 \Gamma^A\Gamma^D)^{-1}E_0 
		\\ &=
			H_0+ F_0
			\Gamma^A\left( I-\Gamma^DG_0\Gamma^A\right)^{-1}\Gamma^D
			E_0
		\\ &\clue{\cref{eq:H0:low-rank}}{=}\,
		\begin{bmatrix}
			U_1^A & F_0\Gamma^A
		\end{bmatrix}\begin{bmatrix}
		I-Y^DY^A & \\ &
				I-\Gamma^D U_1^D(I-Y^AY^D)^{-1}V_1^A\Gamma^A
				\end{bmatrix}^{-1}\begin{bmatrix}
			V^D_1 \\
			\Gamma^DE_0
		\end{bmatrix}
		\\ &\clue{\cref{eq:E0F0:matrix-multiply}}{=}\,
		\begin{multlined}[t]
			\begin{bmatrix}
				U_1^A & \wtd A\Gamma^A 
			\end{bmatrix}
			\begin{bmatrix}
				I& (I-Y^DY^A)^{-1}Y^DV_1^A\Gamma^A
				\\
				& I
			\end{bmatrix}
			\\
			\cdot\begin{bmatrix}
		I-Y^DY^A & \\ &
				I-\Gamma^D U_1^D(I-Y^AY^D)^{-1}V_1^A\Gamma^A
			\end{bmatrix}^{-1}
			\begin{bmatrix}
				I &\\
				\Gamma^DU_1^D(I-Y^AY^D)^{-1}Y^A & I 
			\end{bmatrix}
			\begin{bmatrix}
				V_1^D \\ \Gamma^D \wtd D
			\end{bmatrix}
		\end{multlined}
		\\ &=\,
		\begin{bmatrix}
			U_1^A & \wtd A\Gamma^A 
		\end{bmatrix}
		\begin{bmatrix}
			I-Y^DY^A & -Y^DV_1^A\Gamma^A \\ -\Gamma^DU_1^DY^A & 
			I-\Gamma^DU_1^DV_1^A\Gamma^A
		\end{bmatrix}^{-1}
		\begin{bmatrix}
			V_1^D \\ \Gamma^D\wtd D 
		\end{bmatrix}
		\\ &=\,
		\begin{bmatrix}
			U_1^A &  \wtd A\Gamma^A 
		\end{bmatrix}\left(
			I -\begin{bmatrix}
				T_1^D \\ \Gamma^D U_1^D 
				\end{bmatrix}\begin{bmatrix}
				T_1^A & V_1^A \Gamma^A 
			\end{bmatrix}\right)^{-1}\begin{bmatrix}
			V_1^D\\ \Gamma^D \wtd D 
		\end{bmatrix}
		.
	\end{align*}
	Then examine the recursion.
	\begin{align*}
		X_{t+1}&=\, H_0 + F_0 X_{t}(I-G_0 X_{t})^{-1}E_0 
		\\
		&\clue{\cref{eq:noniter:X:arb}}{=}\, 
		\begin{multlined}[t]
			H_0+ F_0
			\begin{bmatrix}
				U_t^A & \wtd A^t\Gamma^A 
			\end{bmatrix}
			\left(
				I-\begin{bmatrix}
					T_t^D \\ \Gamma^D U_t^D
					\end{bmatrix}\begin{bmatrix}
					T_t^A &  V_t^A  \Gamma^A
			\end{bmatrix}\right)^{-1}\begin{bmatrix}
				V_t^D\\ \Gamma^D \wtd D^t 
			\end{bmatrix}
			\\
			\cdot 
			\left(I-G_0
				\begin{bmatrix}
					U_t^A & \wtd A^t\Gamma^A 
				\end{bmatrix}
				\left(
					I-\begin{bmatrix}
						T_t^D \\ \Gamma^D U_t^D
						\end{bmatrix}\begin{bmatrix}
						T_t^A &  V_t^A  \Gamma^A
					\end{bmatrix}\right)^{-1}\begin{bmatrix}
					V_t^D\\ \Gamma^D \wtd D^t 
			\end{bmatrix}\right)^{-1}
			E_0
		\end{multlined}
		\\ &=
		\begin{multlined}[t]
			H_0+ F_0
			\begin{bmatrix}
				U_t^A & \wtd A^t\Gamma^A 
			\end{bmatrix}
			\\
			\cdot \left(
				I-\begin{bmatrix}
					T_t^D \\ \Gamma^D U_t^D
					\end{bmatrix}\begin{bmatrix}
					T_t^A &  V_t^A  \Gamma^A
					\end{bmatrix} - \begin{bmatrix}
					V_t^D\\ \Gamma^D \wtd D^t 
				\end{bmatrix} G_0
				\begin{bmatrix}
					U_t^A & \wtd A^t\Gamma^A 
				\end{bmatrix}
			\right)^{-1}\begin{bmatrix}
				V_t^D\\ \Gamma^D \wtd D^t 
			\end{bmatrix}
			E_0
		\end{multlined}
		\\ &\clue{\cref{eq:H0:low-rank}}{=}\, 
		\begin{multlined}[t]
		\begin{bmatrix}
			U_1^A & F_0
			\begin{bmatrix}
				U_t^A & \wtd A^t\Gamma^A 
			\end{bmatrix}
		\end{bmatrix}\begin{bmatrix}
		I-Y^DY^A & \\ &
				I-\begin{bmatrix}
					T_t^D \\ \Gamma^D U_t^D
					\end{bmatrix}\begin{bmatrix}
					T_t^A &  V_t^A  \Gamma^A
					\end{bmatrix} - \begin{bmatrix}
					V_t^D\\ \Gamma^D \wtd D^t 
				\end{bmatrix} G_0
				\begin{bmatrix}
					U_t^A & \wtd A^t\Gamma^A 
				\end{bmatrix}
				\end{bmatrix}^{-1}
				\\
				\cdot \begin{bmatrix}
			V^D_1 \\\begin{bmatrix}
				V_t^D\\ \Gamma^D \wtd D^t 
			\end{bmatrix}
			E_0
		\end{bmatrix}
	\end{multlined}
		\\ &\clue{\cref{eq:E0F0:matrix-multiply}}{=} \,
		\begin{multlined}[t]
			\begin{bmatrix}
				U_1^A &	\wtd AU_t^A & \wtd A^{t+1}\Gamma^A 
			\end{bmatrix}
			\begin{bmatrix}
				I&(I-Y^DY^A)^{-1}Y^DV_1^AU_t^A & (I-Y^DY^A)^{-1}Y^DV_1^A\wtd A^t\Gamma^A
				\\
				&I &  
				\\
				&& I
			\end{bmatrix}
			\\
			\cdot\begin{bmatrix}
		I-Y^DY^A & \\ &
				I-\begin{bmatrix}
					T_t^D \\ \Gamma^D U_t^D
					\end{bmatrix}\begin{bmatrix}
					T_t^A &  V_t^A  \Gamma^A
					\end{bmatrix} - \begin{bmatrix}
					V_t^D\\ \Gamma^D \wtd D^t 
				\end{bmatrix} G_0
				\begin{bmatrix}
					U_t^A & \wtd A^t\Gamma^A 
				\end{bmatrix}
				\end{bmatrix}^{-1}
			\\
			\cdot
			\begin{bmatrix}
				I &&\\
				V_t^DU_1^D(I-Y^AY^D)^{-1}Y^A & I & 
				\\
				\Gamma^D\wtd D^tU_1^D(I-Y^AY^D)^{-1}Y^A &  & I 
			\end{bmatrix}
			\begin{bmatrix}
				V_1^D \\	V_t^D\wtd D \\ \Gamma^D \wtd D^{t+1} 
			\end{bmatrix}
		\end{multlined}
		\\ &=\,
		\begin{multlined}[t]
			\begin{bmatrix}
				U_{t+1}^A & \wtd A^{t+1}\Gamma^A 
			\end{bmatrix}
			\\\cdot
			\begin{bmatrix}
				I-Y^DY^A & -Y^DV_1^AU_t^A&-Y^DV_1^A\wtd A^t\Gamma^A
				\\
				-V_t^DU_1^DY^A & I-T_t^DT_t^A-V_t^DU_1^DV_1^AU_t^A & -T_t^DV_t^A\Gamma^A-V_t^DU_1^DV_1^A\wtd A^t\Gamma^A
				\\
				-\Gamma^D\wtd D^tU_1^DY^A & -\Gamma^DU_t^DT_t^A-\Gamma^D\wtd D^tU_1^DV_1^AU_t^A & I-\Gamma^DU_t^DV_t^A\Gamma^A-\Gamma^D\wtd D^tU_1^DV_1^A\wtd A^t\Gamma^A    
			\end{bmatrix}^{-1}
			\\
			\cdot \begin{bmatrix}
				V_{t+1}^D \\ \Gamma^D \wtd D^{t+1} 
			\end{bmatrix}
		\end{multlined}
		\\
		&=\, 
			\begin{bmatrix}
				U_{t+1}^A & \wtd A^{t+1}\Gamma^A 
			\end{bmatrix}
			\left(I- \begin{bmatrix}
					Y^D & 0\\
					V_t^DU_1^D & T_t^D
					\\
					\Gamma^D\wtd D^tU_1^D & \Gamma^DU_t^D 
			\end{bmatrix}
		\begin{bmatrix}
			Y^A & V_1^AU_t^A & V_1^A\wtd A^t\Gamma^A 
			\\
			0 & T_t^A & V_t^A\Gamma^A
		\end{bmatrix}\right)^{-1}
			 \begin{bmatrix}
				V_{t+1}^D \\ \Gamma^D \wtd D^{t+1} 
			\end{bmatrix}
		\\
		&=\,
		\begin{bmatrix}
			U_{t+1}^A  & \wtd A^{t+1} \Gamma^A
			\end{bmatrix}
			\left(I - \begin{bmatrix}
				T_{t+1}^D \\ \Gamma^D  U_{t+1}^D 
				\end{bmatrix}\begin{bmatrix}
				T_{t+1}^A &  V_{t+1}^A \Gamma^A   
			\end{bmatrix} \right)^{-1}\begin{bmatrix}
			V_{t+1}^D\\ \Gamma^D \wtd D^{t+1}
		\end{bmatrix}
		.
		\qedhere
	\end{align*}
\end{proof}

{\small
	\bibliographystyle{SIMAX/siamplain}
\bibliography{SIMAX/fastNARE-min}

\begin{thebibliography}{10}

\bibitem{aboukandilB1985analytical}
{\sc H.~Abou-Kandil and P.~Bertand}, {\em Analytical solution for an open-loop {Stackelberg} game}, IEEE Trans. Automat. Control, 30 (1985), pp.~1222--1224.

\bibitem{aboukandilFIJ2003matrix}
{\sc H.~Abou-Kandil, G.~Freiling, V.~Ionescu, and G.~Jank}, {\em Matrix Riccati Equations in Control and Systems Theory}, Birkh{\"a}uer, Basel, Switzerland, 2003.

\bibitem{antsaklisM2007linear}
{\sc P.~J. Antsaklis and A.~N. Michel}, {\em A linear systems primer}, Birkh{\"a}user, Boston, 2007.

\bibitem{bagby1969interpolation}
{\sc T.~Bagby}, {\em On interpolation by rational functions}, Duke Math. J., 36 (1969), pp.~95--104.

\bibitem{beanOS2010stochastic}
{\sc N.~G. Bean, M.~M. Oreilly, and J.~E. Sargison}, {\em A stochastic fluid model of the operation and maintenance of power generation systems}, IEEE Trans. Power Syst., 25 (2010), pp.~1361--1374.

\bibitem{bellmanW1975introduction}
{\sc R.~Bellman and G.-M. Wing}, {\em An Introduction to Invariant Embedding}, John Wiley, New York, 1975.

\bibitem{bennerBKS2018radi}
{\sc P.~Benner, Z.~Bujanovi{\'c}, P.~K{\"u}rschner, and J.~Saak}, {\em {RADI}: a low-rank {ADI-type} algorithm for large-scale algebraic {Riccati} equations}, Numer. Math., 138 (2018), pp.~301--330.

\bibitem{bennerDKK2017fast}
{\sc P.~Benner, S.~Dolgov, V.~Khoromskaia, and B.~N. Khoromskij}, {\em Fast iterative solution of the {Bethe-Salpeter} eigenvalue problem using low-rank and {QTT} tensor approximation}, J. Comput. Phys., 334 (2017), pp.~221--239.

\bibitem{bennerKS2016lowrank}
{\sc P.~Benner, P.~K{\"u}rschner, and J.~Saak}, {\em Low-rank {Newton-ADI} methods for large nonsymmetric algebraic {Riccati} equations}, J. Franklin Inst., 353 (2016), pp.~1147--1167.

\bibitem{biniIM2012numerical}
{\sc D.~A. Bini, B.~Iannazzo, and B.~Meini}, {\em Numerical Solution of Algebraic Riccati Equations}, vol.~9 of Fundamentals of Algorithms, SIAM Publications, Philadelphia, 2012.

\bibitem{davisH2011university}
{\sc T.~A. Davis and Y.~Hu}, {\em The university of {Florida} sparse matrix collection}, ACM Trans. Math. Software, 38 (2011), p.~Article 1.
\newblock 25 pages.

\bibitem{freilingJA1996global}
{\sc G.~Freiling, G.~Jank, and H.~Abou-Kandil}, {\em On the global existence of solutions to coupled matrix {Riccati} equations in closed-loop {Nash} games}, IEEE Trans. Automat. Control, 41 (1996), pp.~264--269.

\bibitem{freilingJA1999discrete}
{\sc G.~Freiling, G.~Jank, and H.~Abou-Kandil}, {\em Discrete time {Riccati} equations in open-loop {Nash} and {Stackelberg} games}, European J. Control, 5 (1999), pp.~56--66.

\bibitem{ganapol1992investigation}
{\sc B.~D. Ganapol}, {\em An investigation of a simple transport model}, Transport. Theory Statist. Phys., 21 (1992), pp.~1--37.

\bibitem{gonchar1968generalized}
{\sc A.~A. Gonchar}, {\em On a generalized analytic continuation}, Math. USSR-Sb., 5 (1968), pp.~129--140.

\bibitem{gonchar1969zolotarev}
{\sc A.~A. Gonchar}, {\em Zolotarev problems connected with rational functions}, Math. USSR-Sb., 7 (1969), pp.~623--635.

\bibitem{guoCL2020highly}
{\sc Z.-C. Guo, E.~K.-W. Chu, and X.~Liang}, {\em Highly accurate decoupled doubling algorithm for large-scale {M}-matrix algebraic {Riccati} equations}, 2020, \url{https://arxiv.org/abs/2011.00471}.

\bibitem{guoCLL2020decoupled}
{\sc Z.-C. Guo, E.~K.-W. Chu, X.~Liang, and W.-W. Lin}, {\em A decoupled form of the structure-preserving doubling algorithm with low-rank structures}, 2020, \url{https://arxiv.org/abs/2005.08288}.

\bibitem{guoL2023intrinsic}
{\sc Z.-C. Guo and X.~Liang}, {\em The intrinsic {Toeplitz} structure and its applications in algebraic {Riccati} equations}, Numer. Alg., 93 (2023), pp.~227--267.

\bibitem{huangLL2018structurepreserving}
{\sc T.-M. Huang, R.-C. Li, and W.-W. Lin}, {\em Structure-Preserving Doubling Algorithms for Nonlinear Matrix Equations}, vol.~14 of Fundamentals of Algorithms, SIAM, Philadelphia, 2018.

\bibitem{ionescuOW1999generalized}
{\sc V.~Ionescu, C.~Oar\u{a}, and M.~Weiss}, {\em Generalized {Riccati} Theory and Robust Control: A {Popov} Function Approach}, John Wiley \& Sons, Chichester, UK, 1999.

\bibitem{juang1995existence}
{\sc J.~Juang}, {\em Existence of algebraic matrix {Riccati} equations arising in transport theory}, Linear Algebra Appl., 230 (1995), pp.~89--100.

\bibitem{juangC1993iterative}
{\sc J.~Juang and I.~Chen}, {\em Iterative solution for a certain class of algebraic matrix {Riccati} equations arsing in transport theory}, Transport. Theory Statist. Phys., 22 (1993), pp.~65--80.

\bibitem{kremer2003nonsymmetric}
{\sc D.~Kremer}, {\em Non-symmetric Riccati Theory and Noncooperative Games}, no.~30 in Aachener Beitr{\"a}ge zur Mathematik, Wissenschaftverlag Mainz, Aachen, Germany, 2003.

\bibitem{lancasterR1995algebraic}
{\sc P.~Lancaster and L.~Rodman}, {\em Algebraic Riccati Equations}, Oxford University Press, New York, USA, 1995.

\bibitem{latoucheT2009stochastic}
{\sc G.~Latouche and P.~G. Taylor}, {\em A stochastic fluid model for an ad hoc mobile network}, Queueing Syst., 63 (2009), pp.~109--129.

\bibitem{liuXL2022nonlinear}
{\sc C.~Liu, J.~Xue, and R.-C. Li}, {\em On nonlinear matrix equations from the first standard form}, Ann. Math. Sci. Appl., 7 (2022), pp.~169--191.

\bibitem{massoudiOR2016analysis}
{\sc A.~Massoudi, M.~R. Opmeer, and T.~Reis}, {\em Analysis of an iteration method for the algebraic {Riccati} equations}, SIAM J. Matrix Anal. Appl., 37 (2016), pp.~624--648.

\bibitem{morwiki_steel}
{\sc {Oberwolfach Benchmark Collection}}, {\em Steel profile}.
\newblock hosted at {MORwiki} -- Model Order Reduction Wiki, 2005, \url{http://modelreduction.org/index.php/Steel_Profile}.

\bibitem{rogers1994fluid}
{\sc L.~Rogers}, {\em Fluid models in queueing theory and {W}iener-{H}opf factorization of {M}arkov chains}, Ann. Appl. Probab., 4 (1994), pp.~390--413.

\bibitem{SaaKB21-mmess-2.1}
{\sc J.~Saak, M.~K\"{o}hler, and P.~Benner}, {\em {M-M.E.S.S.}-2.1 -- the matrix equations sparse solvers library}, Apr. 2021, \url{https://doi.org/10.5281/zenodo.4719688}.
\newblock see also: \url{https://www.mpi-magdeburg.mpg.de/projects/mess}.

\bibitem{foreestMS2003analysis}
{\sc N.~van Foreest, M.~Mandjes, and W.~Scheinhardt}, {\em Analysis of a feedback fluid model for heterogeneous {TCP} sources}, Stoch. Models, 19 (2003), pp.~299--324.

\bibitem{wangWL2012alternatingdirectional}
{\sc W.-G. Wang, W.-C. Wang, and R.-C. Li}, {\em Alternating-directional doubling algorithm for {M}-matrix algebraic {R}iccati equations}, SIAM J. Matrix Anal. Appl., 33 (2012), pp.~170--194.

\bibitem{xueXL2012accurate}
{\sc J.~Xue, S.~Xu, and R.-C. Li}, {\em Accurate solutions of {M-matrix} algebraic {R}iccati equations}, Numer. Math., 120 (2012), pp.~671--700.

\bibitem{zolotarev1877application}
{\sc E.~I. Zolotarev}, {\em Application of elliptic functions to questions of functions deviating least and most form zero}, Zap. Imp. Akad. Nauk St. Petersburg, 30 (1877), pp.~1--59.
\newblock in Russian.

\end{thebibliography}
}

\end{document}